\documentclass[a4paper]{amsart}
\usepackage{amsmath,amsthm,amssymb}
\usepackage{enumitem}
\usepackage{todonotes}
\usepackage{mathrsfs}
\usepackage{makecell}
\usepackage{xcolor}
\usepackage[pdfencoding=auto,hidelinks]{hyperref}
\usepackage{cleveref}
\usepackage{tikz-cd}
\usepackage[alphabetic,initials,msc-links]{amsrefs}

\newcommand{\dirlim}[1]{\lim_{\rightarrow}}

\newcommand{\FFF}{\mathscr{F}}
\newcommand{\nr}{\operatorname{nr}}

\newcommand{\OO}{\mathcal{O}}
\newcommand{\CCC}{\mathscr{C}}

\newcommand{\tors}{\operatorname{tors}}

\newcommand{\abs}[1]{\left|{#1}\right|}
\newcommand{\set}[1]{\left\{{#1}\right\}}

\renewcommand{\hom}{\operatorname{Hom}}

\newcommand{\ZZ}{\mathbb{Z}}
\newcommand{\QQ}{\mathbb{Q}}
\newcommand{\RR}{\mathbb{R}}
\newcommand{\FF}{\mathbb{F}}
\newcommand{\gp}[1]{\left\langle {#1}\right\rangle}

\newcommand{\im}{\operatorname{im}}

\newcommand{\res}{\operatorname{res}}
\newcommand{\infl}{\operatorname{inf}}

\newcommand{\cores}{\operatorname{cor}}
\newcommand{\sel}[1]{\operatorname{Sel}_{#1}}
\newcommand{\gal}{\operatorname{Gal}}

\newcommand{\pp}{\mathfrak{p}}
\newcommand{\PP}{\mathfrak{P}}

\newcommand{\CC}{\mathbb{C}}

\newcommand{\rank}{\operatorname{rk}}
\newcommand{\places}{\Omega}

\newcommand{\Res}{\operatorname{Res}}
\renewcommand{\epsilon}{\varepsilon}

\newcommand{\cyc}{\operatorname{cyc}}
\newcommand{\avrk}{\operatorname{Avrk}}

\usepackage[OT2,T1]{fontenc}
\DeclareSymbolFont{cyrletters}{OT2}{wncyr}{m}{n}
\DeclareMathSymbol{\Sha}{\mathalpha}{cyrletters}{"58}

\theoremstyle{plain}
\newtheorem{theorem}{Theorem}[section]
\newtheorem{lemma}[theorem]{Lemma}
\newtheorem{cor}[theorem]{Corollary}
\newtheorem{example}[theorem]{Example}
\newtheorem{conj}{Conjecture}
\newtheorem{proposition}[theorem]{Proposition}

\newtheorem{hypothesis}{Hypothesis}

\theoremstyle{definition}
\newtheorem{definition}[theorem]{Definition}
\newtheorem{notation}[theorem]{Notation}

\theoremstyle{remark}
\newtheorem{rem}[theorem]{Remark}

\crefname{section}{\S\!\!}{\S\!\!}
\Crefname{section}{\S\!\!}{\S\!\!}

\creflabelformat{equation}{(#2#1#3)}
\Crefname{equation}{}{}
\crefname{equation}{}{}

\newcommand{\Epsilon}{\mathcal{E}}
\newcommand{\SSS}{\mathscr{S}}
\newcommand{\bigo}[1]{O\left({#1}\right)}
\title{The Failure of Galois Descent for \texorpdfstring{\lowercase{$p$}}{p}-Selmer Groups of Elliptic Curves}
\author{Ross Paterson}
\makeatletter
\hypersetup{
pdftitle={\@title},
pdfauthor={Ross Paterson},
pdfsubject={Number Theory},
pdfkeywords={elliptic curves, arithmetic statistics}
}
\address{School of Mathematics and Statistics, University of Glasgow, University Place, Glasgow, G12 8QQ.}
\email{rosspatersonmath@gmail.com}
\makeatother
\begin{document}
\begin{abstract}
We show that if $F$ is $\QQ$ or a multiquadratic number field, $p\in\set{2,3,5}$, and $K/F$ is a Galois extension of degree a power of $p$, then for elliptic curves $E/\QQ$ ordered by height, the average dimension of the $p$-Selmer groups of $E/K$ is bounded.  In particular, this provides a bound for the average $K$-rank of elliptic curves $E/\QQ$ for such $K$.  Additionally, we give bounds for certain representation--theoretic invariants of Mordell--Weil groups over Galois extensions of such $F$.

The central result is that: for each finite Galois extension $K/F$ of number fields and prime number $p$, as $E/\QQ$ varies, the difference in dimension between the Galois fixed space in the $p$-Selmer group of $E/K$ and the $p$-Selmer group of $E/F$ has bounded average.
\end{abstract}

\maketitle
\section{Introduction}
As $E$ varies amongst elliptic curves over the rational numbers (ordered by height), Bhargava and Shankar \cites{bhargava2Sel} were the first to show that the average rank of the Mordell--Weil group $E(\QQ)$ is bounded.  
It is then natural to ask: for a fixed number field $K$, is the same true of $E(K)$?  Moreover, what dependence does the average rank of $E(K)$ have on $K$?  For multiquadratic extensions $K/\QQ$, an upper bound for the average rank can be derived from the work of Bhargava--Shankar.  With the exception of these bounds, we provide the first known results in this direction.

Let 
\[\Epsilon=\set{(A,B)\in\ZZ^2: \substack{\gcd(A^3,B^2)\textnormal{ is }12^{th}\textnormal{-power free,}\\4A^3+27B^2\neq 0}},\]
which parametrises a set of elliptic curves via the identification $(A,B)\leftrightarrow E_{A,B}:y^2=x^3+Ax+B$.
It is well known, see e.g. \cite{silverman2009arithmetic}*{III.1}, that every elliptic curve defined over the rational numbers is isomorphic to a unique curve in the set of curves parametrised by $\Epsilon$.  The height of $(A,B)\in\Epsilon$, or equivalently of the curve $E_{A,B}$, is defined to be $H(A,B)=H(E_{A,B})=\max\{4\abs{A}^3,27B^2\}$, and for every positive real number $X$, we write $\Epsilon(X)$ for the finite subset of $\Epsilon$ of pairs which have height at most $X$.  Our main result is then:
\begin{theorem}\label{INTRO RANK THM}
	Let $p\in\set{2,3,5}$, $F$ be either $\QQ$ or a multiquadratic number field, and $K/F$ be a Galois $p$--extension.  Then
	\begin{eqnarray*}
	\lefteqn{\limsup_{X\to\infty}\frac{\sum\limits_{(A,B)\in\Epsilon(X)}\rank E_{A,B}(K)}{\#\Epsilon(X)}}
	\\&&\leq
	\begin{cases}
	[K:F]C_2(K/F)+[K:\QQ]\left(C_2(F/\QQ)+\frac{3}{2}\right)&\textnormal{if }p=2\textnormal{ and }F\neq \QQ,\\
	[K:F]\left(C_p(K/F)+\frac{p+1}{p}[F:\QQ]\right)&\textnormal{else,}
	\end{cases}
	\end{eqnarray*}
	where $C_{p}(K/F)$ and $C_{p}(F/\QQ)$ are explicit constants (see \S\ref{subsec:bound shape}).
\end{theorem}
By a multiquadratic number field, we will always mean a number field $F$ which is a finite Galois extension $F/\QQ$ with $\gal(F/\QQ)\cong (\ZZ/2\ZZ)^r$ for some $r>0$.  For a prime number $p$ we say that a Galois field extension $K/F$ is a $p$-extension if $\gal(K/F)$ is a finite $p$-group.

\begin{rem}
	One can obtain stronger bounds in the case that $K/\QQ$ is multiquadratic.  These are obtained from the results of Bhargava and Shankar by computing the average size of the $5$--Selmer group of the Weil restrictions of our $E/\QQ$ from $K$ (see \Cref{prop:multiquadratic good char average}).
\end{rem}

\begin{rem}
	Conditional on a conjecture of Poonen and Rains \cite{Poonen_2012}*{Conjecture 1.1(b)}, the conclusion of \Cref{INTRO RANK THM} holds for every prime number $p$.  In \S\ref{subsec:conjectures of PR} we discuss the consequences of this conjecture to our results.
\end{rem}

\begin{rem}
	This average rank growth compares nicely with Iwasawa--theoretic considerations in $\ZZ_p$--towers above $F$, as we discuss in \Cref{subsec:Iwasawa}.
\end{rem}

\subsection{Galois Descent}
\Cref{INTRO RANK THM} arises, as is the fashion, from a detailed study of statistical properties of Selmer groups (see e.g. \cite{silverman2009arithmetic}*{X.4}).  For a finite Galois extension of number fields $K/F$ and prime number $p$, we study the failure of Galois descent from $K$ to $F$ for $p$-Selmer groups of elliptic curves $E/\QQ$.  That is, we examine the difference
\begin{equation}\label{INTRO eq descent failure}\dim_{\FF_p}\sel{p}(E/K)^{G}-\dim_{\FF_p} \sel{p}(E/F),\end{equation}
where $G=\gal(K/F)$.

In the case that $p\nmid \#G$ this difference is $0$: the finite cohomology groups $H^i(K/F, E(K)[p])$ and their local analogues are trivial, so the inflation--restriction exact sequence yields an isomorphism $\sel{p}(E/F)\cong \sel{p}(E/K)^{G}$ (see also \S\ref{sec:WeilRes} for a more geometric explanation).  In other words, Galois descent does not fail in the ``good characteristic'' case.

The interesting case, that of so-called ``bad characteristic'', is when $p\mid \#G$.  In this case, Galois descent can fail to an arbitrary extent.  Indeed, consider the congruent number curve, which has Weierstrass equation $y^2=x^3-x$, and let $K/\QQ$ be an arbitrary quadratic field.  On one hand, a recent result of Morgan and the author \cite{MorganPaterson2020}*{Theorems 1.1 and 1.3} implies that, for any fixed positive real number $z$, $100\%$ of quadratic twists $E_d$ of $E$ have
\[\dim_{\FF_2}\sel{2}(E_d/K)^G=\dim_{\FF_2}\sel{2}(E_d/K)>z;\]
on the other hand, early work of Heath-Brown \cites{HB1,HB2} shows that more than $99.9\%$ of quadratic twists $E_d$ of $E$ have
\[\dim_{\FF_2}\sel{2}(E_d/\QQ)\leq 6.\]
In particular, in this latter proportion of twists the difference \eqref{INTRO eq descent failure} must be arbitrarily large.  
\begin{rem}
	There are many more examples of this phenomenon:  if $E/\QQ$ is any elliptic curve with full $2$-torsion then the result of Morgan and the author implies the same behaviour for the groups $\sel{2}(E_d/K)$; if additionally $E$ has no rational $4$-isogeny then the same behaviour as above is known to holds for $\sel{2}(E_d/\QQ)$ by work of Kane \cite{kane2013ranks}*{Theorem 3} and Swinnerton-Dyer \cite{SD2008}. 
\end{rem}

The core statistical result in this paper shows that, despite this, the average size of the failure of Galois descent is bounded as we vary over all elliptic curves over the rational numbers (ordered by height).
\begin{theorem}[\Cref{thm:avg dim of fixed space is bdd}]\label{INTRO: gal descent failure is bdd}
	Let $p$ be a prime number, $F$ be a number field and $K/F$ be a finite Galois extension.  Writing $G=\gal(K/F)$, we have that
	\[\limsup_{X\to\infty}\frac{\sum\limits_{(A,B)\in\Epsilon(X)}\abs{\dim_{\FF_p}\sel{p}(E_{A,B}/K)^G - \dim_{\FF_p}\sel{p}(E_{A,B}/F)}}{\#\Epsilon(X)}\leq C_p(K/F),\]
	where $C_{p}(K/F)$ is an explicit constant (see \S\ref{subsec:bound shape}).
\end{theorem}

\subsection{Selmer Ranks}
In the case of a Galois $p$-extension, the $p$-Selmer group is a modular representation of the Galois group.  Appealing to the theory of such, we use \Cref{INTRO: gal descent failure is bdd} to bound the average dimension of the full Selmer group.
\begin{theorem}[\Cref{cor:p extensions of multiquadratics}]\label{INTRO: full p selmer thm}
	Let $p\in\set{2,3,5}$, $F$ be either $\QQ$ or a multiquadratic number field, and $K/F$ be a Galois $p$--extension.  Then
	\begin{eqnarray*}
	\lefteqn{\limsup_{X\to\infty}\frac{\sum\limits_{(A,B)\in\Epsilon(X)}\dim_{\FF_p}\sel{p}(E_{A,B}/K)}{\#\Epsilon(X)}}
	\\&&\leq
	\begin{cases}
	[K:F]C_2(K/F)+[K:\QQ]\left(C_2(F/\QQ)+\frac{3}{2}\right)&\textnormal{if }p=2\textnormal{ and }F\neq \QQ,\\
	[K:F]\left(C_p(K/F)+\frac{p+1}{p}[F:\QQ]\right)&\textnormal{else,}
	\end{cases}
	\end{eqnarray*}
	where $C_{p}(K/F)$ and $C_{p}(F/\QQ)$ are explicit constants (see \S\ref{subsec:bound shape}).
\end{theorem}
It is from this result, and the usual inclusion $E(K)/pE(K)\subseteq \sel{p}(E/K)$, that we obtain \Cref{INTRO RANK THM}.
\begin{example}
	The bounds obtained in \Cref{INTRO: full p selmer thm} are typically rather large.  Let $K/\QQ$ be the splitting field of $x^{10} - 35 x^{6} + 130 x^{4} + 160$, so that the Galois group $\gal(K/\QQ)$ is isomorphic to $D_{10}$ the dihedral group of order $10$.  In this case, $F=\QQ(\sqrt{-10})$ is a multiquadratic field contained in $K$, and $K/F$ is a degree $5$ extension, so we can apply \Cref{INTRO: full p selmer thm} with $p=5$.   We can compute that $C_5(K/F)\leq8.36$.  Thus we have that the average dimension of $5$-Selmer groups over $K$ of elliptic curves over $\QQ$ is less than $54$, and in particular the same bound is true of the average rank of the Mordell--Weil groups $E(K)$.
\end{example}

\subsection{Mordell--Weil Lattices}
We deduce some representation--theoretic information about the ``free part'' of Mordell--Weil groups from \Cref{INTRO: gal descent failure is bdd}.  Specifically, for an elliptic curve $E/\QQ$ and number field $K$ write $\Lambda(E/K)$ for the so--called Mordell--Weil lattice, that is, the quotient of the group $E(K)$ by its torsion subgroup
\[\Lambda(E/K)=E(K)/E(K)_{\tors}.\]
For a finite Galois extension $K/F$, writing $G=\gal(K/F)$, $\Lambda(E/K)$ is a $\ZZ$-free $\ZZ[G]$-module, i.e. a $\ZZ[G]$-module which is free as an abelian group; we refer to such modules as $\ZZ[G]$-lattices.

The integral representation theory of finite groups is more delicate than representation theory over fields.  For example if for some prime number $p$ there is a $p$--Sylow subgroup of $G$ which is not cyclic of order at most $p^2$ then there are infinitely many isomorphism classes of indecomposable $\ZZ[G]$-lattices \cite{curtis1981methods}*{Theorem 33.6}.  Moreover, the unique decomposition of representations, which holds over fields as a result of the Krull--Schmidt--Azumaya theorem \cite{curtis1981methods}*{Theorem 6.12}, does not generally hold for $\ZZ[G]$--lattices, so the na\"ive notion of multiplicity of indecomposable sublattices need not be well defined.

We begin by providing a suitable notion of ``multiplicity'' for a $\ZZ[G]$--lattice $\Lambda$ (see \Cref{def:multiplicity}) inside of $\Lambda(E/K)$, which we denote by $e_{\Lambda}(K/F;E)$.  We then provide a bound for the average of $e_{\Lambda}(K/F;E)$, so long as $\Lambda$ satisfies a local condition somewhere and $F$ is a contained in a multiquadratic field.

\begin{theorem}[\Cref{cor:p=235 general MW result}]\label{INTRO:general p MWlattice result}
	Let $p\in\set{2,3,5}$, $F$ be either $\QQ$ or a multiquadratic number field, and $K/F$ be a finite Galois extension.  Writing $G=\gal(K/F)$, we have that for every $\ZZ [G]$-lattice $\Lambda$ such that $\dim_{\FF_p}(\Lambda/p\Lambda)^G\geq 1$,
	\begin{eqnarray*}
	\lefteqn{\limsup_{X\to\infty}\frac{\sum\limits_{(A,B)\in\Epsilon(X)}e_\Lambda(K/F;E_{A,B})}{\#\Epsilon(X)}}
	\\&&\leq\frac{1}{\dim_{\FF_p}(\Lambda/p\Lambda)^G}\cdot
	\begin{cases}
	C_2(K/F)+[F:\QQ]\left(C_2(F/\QQ)+\frac{3}{2}\right)&\textnormal{if }p=2\textnormal{ and }F\neq \QQ,\\
	C_p(K/F)+\frac{p+1}{p}[F:\QQ]&\textnormal{else,}
	\end{cases}
	\end{eqnarray*}
	where $C_{p}(K/F)$ and $C_{p}(F/\QQ)$ are explicit constants (see \S\ref{subsec:bound shape}).
\end{theorem}
For example, if $G$ is a $p$-group then, by the orbit stabiliser theorem, for every $\ZZ[G]$-lattice $\Lambda$ we have $\dim_{\FF_p}(\Lambda/p\Lambda)^G\geq 1$.  Of course, in this case these multiplicities can already be shown to be bounded by applying \Cref{INTRO RANK THM}.
\begin{rem}\label{rem:rank thm sophisticated app for lattices}
	Many lattice multiplicities can already be bounded using \Cref{INTRO RANK THM}, even when $K/\QQ$ is not of the correct form for direct application.  If $K/\QQ$ is Galois with group $G$, then if one can choose a normal subgroup $N\leq G$ for which $\Lambda^N\neq 0$ then we can track the multiplicity by passing to the associated subfield.  More precisely, in this setting if $L=K^N$ then for each $E/\QQ$,
	\[e_\Lambda(K/\QQ;E)\leq \frac{\rank E(L)}{\rank\Lambda^N}.\]
	If the subextension $L/\QQ$ is of the correct form for \Cref{INTRO RANK THM}, i.e. $G/N$ is an extension of a (possibly trivial) elementary abelian $2$-group by a $p$-group, then we can still bound the multiplicity with \Cref{INTRO RANK THM}.  
\end{rem}
In light of the above, one may ask whether \Cref{INTRO:general p MWlattice result} is a formal consequence of \Cref{INTRO RANK THM}.  The following example demonstrates that this is not the case.

\begin{example}\label{ex:MWLatticeExample}
	Let $K/\QQ$ be a finite Galois extension with Galois group $G\cong\FF_5\rtimes\FF_5^\times$.  Let $\pp\triangleleft\ZZ[\zeta_5]$ be the prime ideal lying over $5$ in the ring of integers of the $5^{th}$ cyclotomic field, upon which $\FF_5$ acts by multiplication by $\zeta_5$ and $\FF_5^\times$ acts as $\gal(\QQ(\zeta_5)/\QQ)$.  It is elementary to check that the actions above induce the structure of a $\ZZ[G]$--lattice on $\pp$.  Since $\dim_{\FF_5}(\pp/5\pp)^G=1$, we can bound the multiplicity using \Cref{INTRO:general p MWlattice result}.  

	However, \Cref{INTRO RANK THM} does not bound this multiplicity, even via the sophisticated application in \Cref{rem:rank thm sophisticated app for lattices}, since the action of every non-trivial normal subgroup $N\leq G$ on $\pp$ is without fixed points.

	Our method does not allow us to bound the multiplicity of the lattice $\ZZ[\zeta_5]$ with the analogous action of $G$: this lattice has no $G$-fixed space, so for every prime number $p$ we have that 
	\[(\Lambda/p\Lambda)^G\cong H^1(G,\ZZ[\zeta_{5}])[p],\] 
	and one can easily compute that $H^1(G,\ZZ[\zeta_{5}])=0$.  In particular, \Cref{INTRO:general p MWlattice result} does not allow us to bound the average multiplicity of $\QQ(\zeta_5)$ as an irreducible subrepresentation inside of $E(K)\otimes \QQ$.
\end{example}

\subsection{Interaction with the Poonen--Rains Heuristics}\label{subsec:conjectures of PR}
\Cref{INTRO RANK THM,INTRO:general p MWlattice result,INTRO: full p selmer thm} all depend on $p$ being a small prime.  However this is an artefact of our current state--of--the--art, rather than an indication of special behaviour.  The following is a well known conjecture in the literature.
\begin{conj}[\cite{Poonen_2012}*{Conjecture 1.1(b)}]\label{conj:PoonenRains}
For each prime number $p$, the average of $\#\sel{p}(E/\QQ)$ over all $E/\QQ$ is $p+1$.
\end{conj}
\Cref{conj:PoonenRains} is known to be true already if $p\in\set{2,3,5}$, via the works of Bhargava and Shankar \cites{bhargava2Sel,bhargava3Sel,bhargava5Sel}, and indeed we use this to obtain our unconditional bounds.  More specifically, \Cref{conj:PoonenRains} predicts that for every prime number $p$ the average of $\dim_{\FF_p}\sel{p}(E/\QQ)$ is at most $\tfrac{p+1}{p}$.

In fact, for $p\in\set{2,3,5}$, Bhargava and Shankar proved that the conjectural average in \Cref{conj:PoonenRains} is also true in the family of all elliptic curves $E/\QQ$ satisfying finitely many congruence conditions (and indeed infinitely many, assuming some technical conditions).  In light of this, we do not think it unreasonable to expect the average in \Cref{conj:PoonenRains} to hold true for the family of all $E/\QQ$ which satisfy a fixed finite number of congruence conditions.  We mark this as a hypothesis for using later below: 
\begin{hypothesis}\label{hyp:PR for cong conds}
	Let $\tilde{\Epsilon}\subseteq \Epsilon$ be a subset defined by finitely many congruence conditions, and for every positive real number $X$ write $\tilde{\Epsilon}(X)=\Epsilon(X)\cap \tilde{\Epsilon}$.  For each prime number $p$, the average of $\#\sel{p}(E/\QQ)$ for $E\in\tilde{\Epsilon}(X)$ goes to $p+1$ as $X\to\infty$.
\end{hypothesis}

\begin{theorem}[\Cref{cor:p extensions of multiquadratics}, \Cref{cor:p=235 general MW result}]\label{INTRO THM: POONEN RAINS IMPLIES ALL P}
	Assuming \Cref{hyp:PR for cong conds}, the conclusions of \Cref{INTRO RANK THM,INTRO:general p MWlattice result,INTRO: full p selmer thm} hold for every prime number $p$. 
\end{theorem}
\begin{rem}
	In fact, the work in \cite{Poonen_2012} predicts an exact summation for the average value of $\dim_{\FF_p}\sel{p}(E/F)$.  We have opted to work with the average size and then bound the average rank by elementary estimates so as to match up with what is currently known.  This does not affect the growth in $p$ of the bounds.
\end{rem}

\subsection{Bound Shape}\label{subsec:bound shape}
The bounds in \Cref{INTRO RANK THM,INTRO: gal descent failure is bdd,INTRO:general p MWlattice result,INTRO: full p selmer thm} all depend on the constants $C_p(K/F)$ and $C_p(F/\QQ)$.  We now comment on their behaviour.  Explicitly, for a prime number $p$ and finite Galois extension of number fields $K/F$,

\[C_p(K/F)=2 \omega_F(6p\Delta_{K})+[F:\QQ]+\delta_2(p)r_1(F) + 
		2\sum\limits_{\substack{\ell \textnormal{ prime}\\\ell\nmid 6p\Delta_{K}}}\omega_F(\ell)\frac{2\ell^8-\ell^7-1}{\ell^{10}-1},\]
where: $\delta_2(p)=1$ if $p=2$ and $\delta_2(p)=0$ otherwise; for an integer $n$, $\omega_F(n)$ is the number of prime ideals of $F$ which divide the ideal generated by $n$ over the integers of $F$; $r_1(F)$ is the number of real embeddings of $F$; and $\Delta_{K}$ is the discriminant of $K$.

This implies some asymptotic bounds for the growth in average ranks of elliptic curves over extension fields.  To ease notation somewhat, for each number field $K$ write
\[\avrk(K):=\limsup_{X\to\infty}\frac{1}{\#\Epsilon(X)}\sum_{(A,B)\in\Epsilon(X)}\rank E_{A,B}(K),\]
for the average rank of the $K$--points on elliptic curves $E/\QQ$.  Then for $K/F$ and $p$ as in \Cref{INTRO RANK THM}, 
	\begin{equation}\label{eq:INTRO RANK GROWTH}
	\avrk(K)\ll [K:\QQ]\omega_\QQ(\Delta_{K}),
	\end{equation}
where the implied constant is absolute.  Moreover, as in \Cref{INTRO THM: POONEN RAINS IMPLIES ALL P}, under the Poonen--Rains heuristics the same holds if we allow $p$ to be chosen from the set of all prime numbers.  

One example application is the growth of ranks in towers of number fields with restricted ramification.

\begin{example}[Fixed base field $F$]\label{example:rank asymptotic}
Let $F$ be $\QQ$ or a fixed multiquadratic number field, and let $F^{(p)}$ be a pro--$p$ extension of $F$ that ramifies only above rational primes in a fixed finite set $S$.  For each finite Galois extension $K/F$ such that $K\subseteq F^{(p)}$, write
\begin{align*}
\operatorname{Avrk}(K)&=\limsup_{X\to\infty}\frac{1}{\Epsilon(X)}\sum_{(A,B)\in\Epsilon(X)}\rank E(K).
\end{align*}
Then the asymptotic in \eqref{eq:INTRO RANK GROWTH} shows that
\[\operatorname{Avrk}(K)\ll [K:F],\]
that is, the average $K$-rank grows at most linearly in the degree of the extension.  In particular this holds if $F^{(p)}$ is the limit of a ray class field tower, or if $F^{(p)}$ is a $\ZZ_p$--extension.
\end{example}

\noindent We can also provide a uniform bound for average ranks over infinitely many extensions.

\begin{example}
	\Cref{INTRO RANK THM} implies that there are infinitely many $S_3$ number fields $K$ for which
	\begin{equation}\label{eq:dontordernumberfieldsbyECrank}
	\avrk(K)< 65.
	\end{equation}
	Indeed, for each prime number $\ell$ take $K_\ell$ to be the splitting field of $X^3-\ell$.  These are cubic extensions of their shared quadratic subfield $F=\QQ(\zeta_3)$, so we compute that if $\ell\equiv 2\mod 3$ then $C_3(K_\ell/F)\leq 8.44$; thus \eqref{eq:dontordernumberfieldsbyECrank} holds with $K=K_\ell$.
\end{example}
\begin{rem}
	Although we can often obtain uniform bounds for average ranks over infinitely many extensions with Galois group isomorphic to some fixed $G$, we cannot use these methods to obtain a bound which works for a positive proportion of such extensions.  Indeed, any sensible ordering of such extensions would see the number of ramified primes grow, which in turn causes our bound to grow.
\end{rem}

\subsection{Comparison to Iwasawa Theory}\label{subsec:Iwasawa}
Strict rank growth control has been observed and predicted for fixed elliptic curves in a few cases; we now show some examples of this and discuss the relationship with our results.  For the duration of this section, we fix a prime number $p$.  If $p\geq 7$ then we also assume \Cref{hyp:PR for cong conds}.  Recasting \Cref{INTRO RANK THM}, as in \Cref{example:rank asymptotic}, we obtain a bound for rank growth in $\ZZ_p$--extensions.

\begin{cor}\label{cor:Iwasawa extns}
	Let $F$ be $\QQ$ or a multiquadratic number field, and let $F_\infty/F$ be a $\ZZ_p$--extension.  For each integer $n\geq 1$, let $F_n$ be the intermediate field $F\subseteq F_n\subseteq F_\infty$ such that $\gal(F_n/F)\cong \ZZ/p^n\ZZ$.  Then for every integer $n\geq 1$
	\[\avrk(F_n)\ll p^n,\]
	where the implied constant is computable and depends only on the choice of base field $F$.
\end{cor}

We now compare this result with some conjectures and results in the literature for fixed elliptic curves.

Recall that the cyclotomic $\ZZ_p$--extension of a number field $F$ is the unique subfield $F_{p-\cyc}\subseteq \bigcup_{n\geq 1}F(\zeta_{p^n})$ such that $\gal(F_{p-\cyc}/F)\cong \ZZ_p$.  Work of Kato \cite{MR2104361} and Rohrlich \cite{MR735333} (see also \cite{MR1860044}*{Theorem 1.4}) shows the following: for each elliptic curve $E/\QQ$, there is an integer $C_{E}$ such that for all subfields $K\subseteq \QQ_{p-\cyc}$ we have $\rank(E(K))\leq C_{E}$

The author is not aware of any reason to expect these $C_{E}$ to be uniformly bounded across all $E/\QQ$, in fact there is substantial debate in the area on whether even the ranks of the rational points $E(\QQ)$ are even uniformly bounded (see \cite{MR3985613}*{\S3} for a historical survey). Moreover, prior to this work it appeared unclear whether, for example, the curves of height at most $X$ could have $C_{E}$ of order $\exp\exp(X)$ and typically attain said maximum at low levels of the tower $\QQ_{p-\cyc}/\QQ$.  \Cref{cor:Iwasawa extns} shows that the hypothetical behaviour of $C_{E}$ above cannot possibly occur.

Fix an imaginary quadratic field $F/\QQ$, then for a general $\ZZ_p$--extension $F_\infty/F$, there is the growth conjecture of Mazur \cite{MR804682}*{\S18}, as extended by Lei and Sprung \cite{MR4093883}*{Conjecture 1.2}, which claims:  if $E/\QQ$ has good reduction at $p$, and $F_\infty$ is not the anticyclotomic extension, then there is an integer $C_{E,F}$ such that for all intermediate fields $F_n$ (as in \Cref{cor:Iwasawa extns}), we ought to have $\rank(E(F_n))\leq C_{E,F}$.  As in the $\QQ_{p-\cyc}$ case above, \Cref{cor:Iwasawa extns}  shows that statistically this $C_{E,F}$ cannot behave wildly as $E/\QQ$ varies.  Moreover, this conjecture only accounts for the $E/\QQ$ with \textit{good reduction} at $p$, which excludes a positive proportion of elliptic curves.  \Cref{cor:Iwasawa extns} suggests that, at least on average, there should not be overly fast growth of ranks for the elliptic curves with bad reduction at $p$.

We now consider the case that $F_\infty/F$ is the anticyclotomic extension, which is characterised by its being dihedral over $\QQ$. The growth number proposition \cite{MR804682}*{\S18} shows that if $E/\QQ$ has good ordinary reduction at $p$, and the N\'eron fibre of $E$ is geometrically connected at every place $v\mid p$ at which the extension $F_\infty/F$ splits infinitely, then for each layer $F_n$ (as in \Cref{cor:Iwasawa extns}), we must have
\[\rank(E(F_n))=a(E,F_\infty/F)p^n+e_n(E),\]
where $\set{e_n(E)}_{n\geq 1}$ is a bounded sequence of integers associated to $E$, and $a(E, F_\infty/F)$ is a fixed growth constant (independent of $n$).  

The growth number conjecture of Mazur (\cite{MR804682}*{\S18 Growth Number Conjecture}) predicts that (for $E/\QQ$ as in the growth number proposition):
	\[a(E,F_\infty/F)=\begin{cases}
		0&\textnormal{if }w_E=1,\\
		1&\textnormal{if }w_E=-1\textnormal{ and }E\textnormal{ does not have CM by }F,\\
		2&\textnormal{if }w_E=-1\textnormal{ and }E\textnormal{ has CM by }F.
	\end{cases}\]
Note that the condition $w_E=-1$ is conjectured to hold for $50\%$ of $E/\QQ$, and is known to hold for at least $27.5\%$ of $E/\QQ$ by \cite{bhargava5Sel}*{Theorem 6}.  The additional stipulations on $E$ and its N\'eron model should again be positive proportion (and for large $p$ this proportion tends towards $100\%$). In particular $a(E,F_\infty/F)>0$ is expected to hold for a positive proportion of $E/\QQ$, and so we should expect from this conjecture that there is \textit{at least} linear growth of average ranks in the degree of the extension. That is, if $F_\infty/F$ is the anticyclotomic extension and for each $n\geq 1$, $F_n$ is the $n$th layer of $F_\infty/F$ as in \Cref{cor:Iwasawa extns}, we should expect
\[\avrk(F_n)\gg p^n.\]
\Cref{cor:Iwasawa extns} shows that, in fact, this is not just a lower bound but is the best possible asymptotic behaviour.

\subsection{Outline}
In \S\ref{sec:WeilRes} we review some well known properties of the Weil restriction of scalars.  We then use these properties to obtain bounds for average dimensions of Selmer groups in ``good characteristic'' over multiquadratic fields.

In \S\ref{sec:local copmputations} we compute local norm indices for elliptic curves over unramified extensions, extending work of Kramer \cite{kramer1981arithmetic}, which may be of independent interest.  

In \S\ref{sec:CoresSelmer} we recall and extend certain results and definitions from \cite{MorganPaterson2020} to the setting of interest, and define the genus theory invariant of an elliptic curve with respect to a Galois extension and prime number $p$, which will represent an upper bound for the size of the obstruction to Galois descent.  We then use this in \S\ref{sec:GalDescent} to obtain \Cref{INTRO: gal descent failure is bdd}.

Following this, in \S\ref{sec:Selmer Ranks} and \S\ref{sec:MW app} respectively, we use \Cref{INTRO: gal descent failure is bdd} to prove \Cref{INTRO: full p selmer thm} and \Cref{INTRO:general p MWlattice result}.  At the end of \S\ref{sec:MW app} we also provide a family of examples which generalise \Cref{ex:MWLatticeExample}.
\subsection{Limitations and Extensions}
Whilst our family of curves is that of all elliptic curves over $\QQ$, one should be able to use these methods to obtain similar results for similar sets of elliptic curves over a fixed number field ordered by height.  We have opted not to do this here, since to do so requires choosing a way to extend the definition of the set $\Epsilon$ from $\QQ$ to a number field.  If the ideal class group of this number field is non-trivial then there can be more than one such parametrisation, so the question of how to parametrise ``all elliptic curves'' is nuanced.

\subsection{Acknowledgements}
The author would like to thank Alex Bartel for countless helpful comments and suggestions.  We would also like to thank Efthymios Sofos for helpful comments on an earlier version of this paper.  We are grateful to the anonymous referee for their thorough reading of the article, and their insight on the flexibility of the works of Bhargava--Shankar, which prompted the choice of height in the article and also \Cref{rem:referee comment}.  Throughout this work, the author was supported by a PhD scholarship from the Carnegie Trust for the Universities of Scotland. 

\subsection{Notation and Conventions}
For a field $F$ of characteristic $0$, we write $\bar{F}$ for a (fixed once and for all) algebraic closure of $F$, and denote its absolute Galois group by $G_F=\gal(\bar{F}/F)$.  By a $G_F$-module $M$ we mean a discrete abelian group $M$ on which $G_F$ acts continuously, and for each $i\geq 0$ we write $H^i(F, M)$ as a shorthand for the continuous cohomology groups $H^i(G_F, M)$.  If moreover $M$ is $p$-torsion for some prime number $p$ then we say that $M$ is an $\FF_p[G_F]$-module, and for $V\subseteq H^i(F, M)$ we write $\dim V$ for the $\FF_p$-dimension of $V$.  For such $M$, we define the \textit{dual} of $M$ to be 
\[M^*:= \hom(M, \boldsymbol{\mu}_p),\]
where $\boldsymbol{\mu}_p$ is the $G_F$-module of $p$\textsuperscript{th} roots of unity in $\bar{F}$.  This is an $\FF_p[G_F]$-module with action given as follows:  for $\sigma\in G_F$, $\phi\in M^*$ and $m\in M$,
	\[{}^\sigma\phi(m)=\sigma\phi(\sigma^{-1}m).\]
For $i\geq 0$, if $L/F$ is a finite extension we denote the corresponding restriction and corestriction maps by
	\[\res_{L/F}:H^i(F, M)\to H^i(L,M)\]
and
	\[\cores_{L/F}:H^i(L, M)\to H^i(F,M),\]
respectively.

For a number field $F$, we write $\places_F$ for the set of places of $F$ and for each $v\in\places_F$ we write $F_v$ for the completion of $F$ at $v$.  For each $v\in\places_F$ we fix (once and for all) an embedding $\bar{F}\hookrightarrow\bar{F}_v$, and so an inclusion $G_{F_v}\subseteq G_F$.  Thus each $G_F$-module $M$ is naturally a $G_{F_v}$-module and moreover when $v$ is non-archimedean (finite), we denote by $F_v^{\nr}$ the maximal unramified extension of $F_v$, and write
	\[H^1_{\nr}(F_v, M)=\ker\left(H^1(F_v, M)\stackrel{\res}{\longrightarrow} H^1(F_v^{\nr}, M)\right)\]
for the subgroup of unramified classes.

For a number field $F$, an elliptic curve $E/F$ and a finite place $v\in\places_F$, when we describe the reduction type of $E$ at $v$ we are implicitly referring to the type of $E$ in the Kodaira--N\'eron classification (see e.g. \cite{silverman1994advanced}*{IV Theorem 8.2}).  

By an arithmetic function, we mean a function $f:\ZZ\to \CC$ such that for each $n\in\ZZ$ we have $f(-n)=f(n)$.  We denote by $\mu$ the M\"obius function and by $\gcd$ the greatest common divisor function, each extended from $\mathbb{N}$ to $\ZZ$ by composition with the archimedean absolute value.  For arithmetic functions $f$ and $g$, we denote by $f*g$ the Dirichlet convolution of the two, i.e. for each $n\in \ZZ$
	\[(f*g)(n):=\sum_{d\mid n}f(d)g(n/d),\]
where the sum is over positive divisors of $n$.  We say that an arithmetic function $f$ is multiplicative if for coprime integers $m,n\in \ZZ$ we have that $f(mn)=f(m)f(n)$.

For each prime number $\ell$ we write $v_\ell$ for the normalised valuation on $\QQ_\ell$, i.e. the unique valuation such that $v_\ell(\ell)=1$.

\section{Good Characteristic:  Weil Restriction}\label{sec:WeilRes}
For the duration of this section, fix a finite Galois extension of number fields $K/F$ and an elliptic curve $E/\QQ$, and write $G=\gal(K/F)$.  We begin in \S\ref{subsec:twists of EC} with expository material on twists of elliptic curves and the Weil restriction.  In \S\ref{subsec:selmer good char} we then go on to survey some results on $p$-Selmer groups in extensions of degree coprime to $p$.  This material is closely related to, and inspired by, that appearing in \cite{mazur2007finding}*{\S3}.  Finally, in \S\ref{subsec: avg selmer ranks in MQ extensions}, we explain how this material allows us to extend the results of Bhargava and Shankar \cites{bhargava3Sel,bhargava5Sel} on the average dimension of $3$- and $5$-Selmer groups over $\QQ$ to a bound for the average dimension of $3$- and $5$-Selmer groups over any multiquadratic number field.

\subsection{Twists of Elliptic Curves}\label{subsec:twists of EC}
As in Milne \cite{MR330174}*{\S2} (see also \cite{mazur2007twisting}), there is a general construction of twists of powers of an elliptic curve, which we now recall. 

\begin{definition}  \label{milnes forms}
Let $n\geq 1$. To each matrix  $M=(m_{i,j})$ in $\textup{Mat}_n(\ZZ)$ we can associate an endomorphism of $E^n$ given by
\[(P_1,...,P_n)\longmapsto \bigg(\sum_{j=1}^nm_{1,j}P_j , ..., \sum_{j=1}^nm_{n,j}P_j\bigg).\]
In this way we view $\textup{GL}_n(\ZZ)$ as a subgroup of $\textup{Aut}_{F}(E^n)$. Now suppose that $\Lambda$ is a free rank-$n$ $\ZZ$-module equipped with a continuous $G_F$-action. Choosing a basis for $\Lambda$ gives rise to a homomorphism
\[\rho_\Lambda: G_F\longrightarrow \textup{GL}_n(\mathbb{Z}),\]
which we view as a $1$-cocycle valued in $\textup{Aut}_{\bar{F}}(E^n)$. The class of $\rho_\Lambda$ in $H^1(F,\textup{Aut}_{\bar{F}}(E^n))$ does not depend on the choice of basis. Associated to this cocycle class is a twist of $E^n$, which we denote $\Lambda \otimes E$. This is an abelian variety over $F$ of dimension $n$, equipped with a $\bar{F}$-isomorphism 
$\varphi_\Lambda:E^n\rightarrow \Lambda \otimes E$
 satisfying $\varphi_\Lambda^{-1}\varphi_\Lambda^\sigma=\rho_\Lambda(\sigma)$
for all $\sigma \in G_F$. 
\end{definition}

The Weil restriction of $E$ can now be defined as a specific example of such a twist.

\begin{definition}\label{weil_restriction}
The Weil restriction of $E$ from $K$ to $F$ is the abelian variety 
\[\Res_{K/F}E=\ZZ[G]\otimes E.\]
\end{definition}

\begin{rem}
The Weil restriction $\textup{Res}_{K/F}E$ is classically defined as the unique scheme over $F$ representing the functor on $F$-schemes 
\[T\longmapsto E(T\times_{F}K).\]
As in \cite{mazur2007twisting}*{Theorem 4.1}, this is equivalent to the construction given above.
\end{rem}

\subsection{Selmer Groups in Good Characteristic}\label{subsec:selmer good char}
Here we remark on the structure of $n$-Selmer groups in the case that $n$ is coprime to $\#G$, the case of so--called ``good characteristic''.  In this case, the $n$-Selmer group splits as a sum over twists of $E$.  This can be viewed as a finite-level explication of the results in \cite{mazur2007finding}*{\S3}, where similar results are shown for Pontrjagin dual $p^{\infty}$--Selmer vector spaces without our restriction on $p$.

\begin{lemma}\label{lem:Weil restriction descent of Selmer}
	For every positive integer $n$, not necessarily coprime to $\#G$, 
	\begin{enumerate}[label=(\roman*)]
	\item\label{enum:WR1} there is a natural isomorphism of $\ZZ[G_F]$--modules
	\[\Res_{K/F}E[n]\cong \ZZ[G]\otimes_{\ZZ}E[n],\]
	where $\sigma\in G_F$ acts on the right hand side diagonally,
	\item\label{enum:WR2} the above isomorphism induces an isomorphism of $\ZZ[G]$--modules
	\[\sel{n}(\Res_{K/F}E/F)\cong \sel{n}(E/K),\]
	where the action of $G$ on the left hand side is induced by the action of $G$ on $\ZZ[G]$ by left multiplication.
	\end{enumerate}	
\end{lemma}
\begin{proof}
	\ref{enum:WR1} is found in \cite{mazur2007twisting}*{Theorem 2.2(ii)}, see also \cite{MR330174}*{\S1(a)}.	For \ref{enum:WR2}, we give an analogous argument to that in \cite{mazur2007finding}*{proof of Proposition 3.1(iii)}, see also \cite{MR330174}*{Proof of Theorem 1} for a similar result for Shafarevich--Tate groups.  Indeed, by \ref{enum:WR1}, Shapiro's lemma (see, e.g. \cite{neukirch2013class}*{Theorem 4.9}) provides a $\ZZ[G]$ isomorphism
	\[H^1(F, (\Res_{K/F}E)[n])\cong H^1(K, E[n]),\]
	where the action of $G$ on the left hand side is induced by left multiplication on $\ZZ[G]$ in the isomorphism of $\ref{enum:WR1}$.	 It is then elementary to check that this isomorphism commutes with the corresponding isomorphisms at the local extensions, and thus restricts to one of Selmer groups.
\end{proof}

\begin{definition}
	Let $\rho$ be an irreducible finite dimensional $\QQ[G]$--module.  As in \cite{mazur2007twisting}*{Definition 4.3} we define the twist of $E$ by $\rho$ to be
	\[E_{\rho}=(\QQ[G]_\rho\cap \ZZ[G])\otimes E,\]
	where $\QQ[G]_\rho$ is the $\rho$-isotypic component of $\QQ[G]$, that is, the sum of all left ideals of $\QQ[G]$ isomorphic to $\rho$.
\end{definition}
\begin{example}\label{ex:multiquadratic extensions}
	If $K/F$ is multiquadratic then these twists are extremely concrete.  Since $G$ is an elementary abelian $2$ group, its irreducible representations are order $2$ characters induced by the quadratic subextensions.  Let $\Delta\in F$ be an element such that $F(\sqrt{\Delta})\subseteq K$, and let $\chi_\Delta$ be the corresponding at--most--quadratic character of $G_F$.  Identifying $\chi_\Delta$ with its corresponding one dimensional $\QQ[G]$-module, this construction gives rise to all irreducible finite-dimensional $\QQ[G]$-modules.  Moreover, it is clear that $\QQ[G]_{\chi_\Delta}\cap \ZZ[G]$ is a rank one free abelian group with action of $\sigma\in G$ given by multiplication by $\chi_\Delta(\sigma)$.  In particular, by \cite{mazur2007twisting}*{Theorem 2.2(i)} we obtain that $E_{\chi_\Delta}=E^{(\Delta)}$ is just the usual quadratic twist of $E$ by $\Delta$.
\end{example}

We can then split the $n$-Selmer group of the Weil restriction into those of these twists.  This result is analogous to \cite{mazur2007finding}*{Corollary 3.7}, where they study the Pontrjagin dual Selmer vector spaces.
\begin{proposition}\label{prop:Selmer Descent splits in good char}
	If $n$ is an integer which is coprime to $\#G$, then we have an isomorphism of $\ZZ[G]$-modules
	\[\sel{n}(E/K)\cong \bigoplus_{\rho}\sel{n}(E_\rho/F),\]
	where the sum is over isomorphism classes of irreducible finite dimensional $\QQ[G]$-modules and the action of $G$ on the summands on the right hand side is induced by the action of $G$ on $\QQ[G]_\rho\cap \ZZ[G]$ via the isomorphism in \Cref{lem:Weil restriction descent of Selmer}\ref{enum:WR1}.
\end{proposition}
\begin{proof}
	By \Cref{lem:Weil restriction descent of Selmer} we need only show that $\sel{n}(\Res_{K/F}E/F)$ splits in this way.  The natural map
	\[f:\bigoplus_{\rho}\left(\ZZ[G]\cap \QQ[G]_\rho\right)\to \ZZ[G],\]
	is injective with finite cokernel, so by \cite{mazur2007twisting}*{Theorem 4.5, see also Lemma 2.4} induces an $F$--isogeny
	\[f_E:\bigoplus_{\rho}E_\rho\to \Res_{K/F}E.\]
	Moreover, since the cokernel of $f$ is $\#G$--torsion, the degree of the isogeny $f_E$ must be a divisor of some power of $\#G$ \cite{mazur2007twisting}*{proof of Lemma 2.4} and so coprime to $n$. In particular, $f_E$ induces an isomorphism of $n$--Selmer groups, and moreover since $f_E$ is an $F$--isogeny the isomorphism is one of $\ZZ[G]$--modules.
\end{proof}

\begin{rem}
	In \cite{mazur2007finding}*{Corollary 3.7} the authors do not need to make assumptions about coprimality, since the error that occurs when $p\mid \#G$ contributes an additional torsion module to the $p^\infty$-Selmer groups.  This in turn vanishes when taking the tensor product with $\QQ_p$ to form the Pontrjagin dual Selmer vector space $\hom(\sel{p^\infty}(E/K), \QQ_p/\ZZ_p)\otimes\QQ_p$.
\end{rem}

\subsection{Average Selmer Ranks in Good Characteristic over Multiquadratic Fields}\label{subsec: avg selmer ranks in MQ extensions}  In this subsection, we will restrict our interest to multiquadratic number fields.  We use the Weil restriction as in \S\ref{sec:WeilRes} to give a bound for Selmer ranks in good characteristic using results of Bhargava and Shankar \cites{bhargava3Sel,bhargava5Sel}.  First, we adapt the results of Bhargava--Shankar for quadratic twists.
\begin{proposition}\label{prop:actually Bharghava Shankar}
	For each squarefree integer $D$ and $p\in\set{2,3,5}$, we have
	\[\lim_{X\to\infty}\frac{\sum_{(A,B)\in\Epsilon(X)}\#\sel{p}(E_{A,B}^{(D)}/\QQ)}{\#\Epsilon(X)}=(p+1).\]
	Moreover, assuming \Cref{hyp:PR for cong conds} the conclusion holds for every prime number $p$.
\end{proposition}
\begin{rem}\label{rem:referee comment}
	For $p\in\set{2,3,5}$ this can be seen directly from the methods of Bhargava--Shankar \cite{bhargava2Sel}, since the quadratic twist $E^{(D)}_{A,B}$ has a (possibly not minimal)Weierstrass equation given by $E_{AD^2,BD^3}:y^2=x^3+AD^2x+BD^3$, and their proofs never make use of the minimality condition and work with finitely many congruence conditions.  However, for completeness (and when $p>5$) we provide a proof below.
\end{rem}
\begin{proof}
	Fix a squarefree integer $D$.  Since the quadratic twist $E^{(D)}_{A,B}$ has a (possibly not minimal) Weierstrass equation given by $E_{AD^2,BD^3}:y^2=x^3+AD^2x+BD^3$.  Thus there is a bijection between $\set{E_{A,B}^{(D)}~:~(A,B)\in\Epsilon(X)}$ and the set
	\[\Epsilon_D(X)=\set{(A,B)\in\ZZ^2~:~\substack{
	4\abs{A}^3, 27B^2\leq D^6X;\\
	D^2\mid A,\,D^3\mid B;\\
	4A^3+27B^2\neq 0;\\
	\forall \ell\nmid D \textnormal{ prime, if }\ell^4\mid A \textnormal{ then } \ell^6\nmid B;\\
	\forall \ell\mid D \textnormal{ prime, if }\ell^6\mid A \textnormal{ then } \ell^9\nmid B
	}},\]
	given by identifying $(A,B)\in\Epsilon_D(X)$ with the curve $E_{A,B}$.  We now partition $\Epsilon_D(X)$ into parts, so as to identify with minimal Weierstrass models.  For each pair $(d_1,d_2)$ of positive squarefree integers such that $D=\pm d_1d_2$, we define
	\[\Epsilon_{d_1,d_2}(X)=\set{(A,B)\in\Epsilon(D^6X)~:~\substack{
	4\abs{A}^3, 27B^2\leq \left(\frac{d_1}{d_2}\right)^6X;\\
	4A^3+27B^2\neq 0;\\
	\forall \ell\nmid d_1d_2 \textnormal{ prime, if }\ell^4\mid A \textnormal{ then } \ell^6\nmid B;\\
	\forall \ell\mid d_1 \textnormal{ prime: } \ell^2\mid A,\ \ell^3\mid B\textnormal{, and if } \ell^4\mid A \textnormal{ then } \ell^6\nmid B;\\
	\forall \ell\mid d_2 \textnormal{ prime: if } \ell^2\mid A \textnormal{ then } \ell^3\nmid B.
	}}.\]
	Note that $\Epsilon_{d_1,d_2}(X)\subseteq \Epsilon\left((\tfrac{d_1}{d_2})^6X\right)$, and moreover $\Epsilon_{d_1,d_2}(X)$ parametrises a large family of elliptic curves ordered by na\"ive height in the sense of Bhargava--Shankar \cites{bhargava2Sel,bhargava3Sel,bhargava5Sel}.  Further, we have that
	\[\Epsilon_D(X)=\bigsqcup_{D=\pm d_1d_2}\set{(d_2^4A, d_2^6B)~:~(A,B)\in\Epsilon_{d_1,d_2}(X)},\]
	where the disjoint union is over pairs of squarefree positive integers $d_1,d_2$ satisfying $D=\pm d_1d_2$.  

	Note that for any fixed pair of squarefree positive integers $d_1,d_2$ we have by \cite{bhargava2Sel}*{Theorem 3.17} that
	\begin{align*}
		\lim_{X\to\infty}\frac{\#\Epsilon_{d_1,d_2}(X)}{\#\Epsilon(X)}&
		=
		\left(\prod_{\substack{\ell\mid d_1\\\textnormal{prime}}}\frac{(\ell^2-1)\ell^3+(\ell^3-1)}{\ell^{10}-1}\right)\left(\prod_{\substack{\ell\mid d_2\\\textnormal{prime}}}\frac{\ell^2(\ell^2-1)\ell^6+\ell^2(\ell^3-1)\ell^3}{\ell^{10}-1}\right)
		\\&=d_2^5\left(\prod_{\substack{\ell\mid D\\\textnormal{prime}}}\frac{\ell^5-1}{\ell^{10}-1}\right).
	\end{align*}
	Thus, 
	\begin{align*}
		&\lim_{X\to\infty}\frac{1}{\#\Epsilon(X)}\sum_{(A,B)\in\Epsilon(X)}\#\sel{p}(E_{A,B}^{(D)}/\QQ)\\
		&=\lim_{X\to\infty}\frac{1}{\#\Epsilon(X)}\sum_{(A,B)\in\Epsilon_D(X)}\#\sel{p}(E_{A,B}/\QQ)\\
		&=\lim_{X\to\infty}\frac{1}{\#\Epsilon(X)}\sum_{D=\pm d_1d_2}\sum_{(A,B)\in\Epsilon_{d_1,d_2}(X)}\#\sel{p}(E_{A,B}/\QQ)\\
		&=(p+1)\left(\prod_{\substack{\ell\mid D\\\textnormal{prime}}}\frac{\ell^5-1}{\ell^{10}-1}\right)\sum_{d\mid D}d^5\\
		&=(p+1),
	\end{align*}
	where the penultimate equality follows from the large family average Selmer group sizes in \cites{bhargava2Sel,bhargava3Sel,bhargava5Sel} and the computation above, and the final follows from an elementary identity for power-of-divisor sums.

	Assuming \Cref{hyp:PR for cong conds}, since the families $\Epsilon_{d_1,d_2}(X)$ are defined by finitely many congruence conditions, the argument above holds for all prime numbers $p$.
\end{proof}

We can then relate this to our ordering via elementary estimates, and similarly obtain a bound for the average Selmer rank.
\begin{proposition}\label{prop:Bharghava Shankar}
	For each squarefree integer $D$ and $p\in\set{2,3,5}$, we have
	\[\limsup_{X\to\infty}\frac{\sum_{(A,B)\in\Epsilon(X)}\dim\sel{p}(E_{A,B}^{(D)}/\QQ)}{\#\Epsilon(X)}\leq\frac{(p+1)}{p}.\]
	Moreover, assuming \Cref{hyp:PR for cong conds} the same is true for every prime number $p$.
\end{proposition}
\begin{proof}
	For each $r\geq 0$ we use the inequality $p^r\geq pr$, so for each $E/\QQ$ we have $\dim\sel{p}(E/\QQ)\leq \#\sel{p}(E/\QQ)/p$.  Thus it follows from \Cref{prop:actually Bharghava Shankar}.
\end{proof}

\begin{definition}
	For a squarefree integer $D$, we write $\chi_D$ for the quadratic character of $G_\QQ$ cutting out $\QQ(\sqrt{D})$, and for an abelian group $M$ we write $M^{\chi_D}$ for the discrete $G_\QQ$-module $M$ with action by $\sigma\in G_\QQ$ given by multiplication by $\chi_D(\sigma)\in\set{\pm 1}$.
\end{definition}

\begin{lemma}\label{lem:decomposition of selmer in good char over MQ}
	Let $F$ be a field contained in a multiquadratic number field, write $G=\gal(F/\QQ)$, and let $E/\QQ$ be an elliptic curve.  Then for every odd prime number $p$ there is an isomorphism of $\ZZ[G]$-modules
	\[\sel{p}(E/F)\cong \bigoplus_{D\in\mathcal{Q}(F)}\sel{p}(E^{(D)}/\QQ)^{\chi_D},\]
	where $\mathcal{Q}(F)$ is the set of squarefree integers $D$ such that $\QQ(\sqrt{D})\subseteq F$ and $E^{(D)}$ is the quadratic twist of $E$ by $D$.
\end{lemma}
\begin{proof}
This follows by applying \Cref{prop:Selmer Descent splits in good char} to multiquadratic extensions as in \Cref{ex:multiquadratic extensions}.
\end{proof}

Now we can state an easy statistical consequence of the decomposition in \Cref{lem:decomposition of selmer in good char over MQ}.
\begin{proposition}\label{prop:multiquadratic good char average}
	Let $F$ be either $\QQ$ or a multiquadratic number field.  Then for $p\in\set{3,5}$,
	\[\limsup_{X\to\infty}\frac{\sum\limits_{(A,B)\in\Epsilon(X)}\dim\sel{p}(E_{A,B}/F)}{\#\Epsilon(X)}\leq \frac{p+1}{p}[F:\QQ],\]
	Moreover, assuming \Cref{hyp:PR for cong conds} the same holds for all odd prime numbers $p$.
\end{proposition}
\begin{proof}
Let $p$ be an odd prime number.  Using the decomposition in \Cref{lem:decomposition of selmer in good char over MQ},
\[\dim\sel{p}(E/F)=\sum_{D\in Q(F)}\dim \sel{p}(E^{(D)}/\QQ),\]
where $E^{(D)}$ is the quadratic twist of $E$ by $D$ and $\mathcal{Q}(F)$ is the set of squarefree integers $D$ such that $\QQ(\sqrt{D})\subseteq F$.  The result now follows from \Cref{prop:Bharghava Shankar}, noting that the size of $\mathcal{Q}(F)$ is precisely $[F:\QQ]$.
\end{proof}

\section{Local Computations}\label{sec:local copmputations}
In this section, let $F, \OO_F, v$ be a finite extension of $\QQ_\ell$ for some prime number $\ell$, its ring of integers and normalised valuation, respectively and let $E/F$ be an elliptic curve with multiplicative reduction.  Moreover, let $K/F$ be an unramified extension, let $n$ be its degree and write $N_{K/F}=\sum_{g\in\gal(K/F)}g\in\ZZ[\gal(K/F)]$ for the usual norm element.  We perform some local computations, extending results of Kramer \cite{kramer1981arithmetic} in the case $n=2$.  Specifically, we determine the norm index for such $E$ using the Tate parametrisation (see e.g. \cite{silverman1994advanced}*{V\S3-5}), the properties of which we recall below.

Recall from \cite{silverman1994advanced}*{V Thms 3.1 and 5.3} that there is a unique element $q\in \OO_F$ with $v(q)>0$ such that $E$ is isomorphic over $\bar{F}$ to $\mathbb{G}_m/q^\ZZ$.  We call $q$ the Tate parameter associated to $E$, and fix such an isomorphism and call it the Tate parametrisation.  Moreover, if $E$ has split multiplicative reduction, then we may assume that the Tate parametrisation is defined over $F$.

Let $L/F$ be the unramified quadratic extension, and for each extension $M/F$ define
\begin{align*}
	I(M)&:=\set{x\in (M\cdot L)^\times~:~N_{(M\cdot L)/M}(x)\in q^\ZZ},\\
	I_0(M)&:=\set{x\in (M\cdot L)^\times~:~ N_{(M\cdot L)/M}(x)=1}.
\end{align*}
If $E$ has non-split multiplicative reduction, then the quadratic twist of $E$ by $L$ has split multiplicative reduction, so we may assume that the Tate parametrisation is defined over any field containing $L$.  However, for a finite extension $M/F$ which does not contain $L$, by \cite{silverman1994advanced}*{V Cor. 5.4} the Tate parametrisation over the compositum $M\cdot L$ yields an isomorphism between $E(M)$ and $I(M)/q^\ZZ$.  This isomorphism identifies $E_0(M)$, the points of the connected component of the identity in the N\'eron model of $E$, with $I_0(M)/q^\ZZ$.

\begin{lemma}\label{lem:q is delta_E}
	If $E/F$ has split multiplicative reduction, then the corresponding Tate parameter $q$ satisfies
	\[v(q)=v(\Delta_E),\]
	where $\Delta_E$ is a minimal discriminant for $E/F$.
\end{lemma}
\begin{proof}
	By \cite{silverman1994advanced}*{V Thm 3.1(b)} we have $\Delta_E=q\prod_{n\geq 1}(1-q^n)^{24}$, so the result is immediate.
\end{proof}

\begin{proposition}\label{prop:split mult norm index}
	If $E/F$ has split multiplicative reduction, then 
	\[E(F)/N_{K/F}E(K)\cong \ZZ/\gcd(v(\Delta_E),n)\ZZ,\]
	where $\Delta_E$ is a minimal discriminant for $E/F$.
\end{proposition}
\begin{proof}
	If $E$ has Tate parameter $q\in\OO_F$ then, since the Tate parametrisation is defined over $F$, we have a commutative square
	\[\begin{tikzcd}
	E(K)\arrow[r, "\sim"]\arrow[d, "N_{K/F}"]&K^\times/q^\ZZ\arrow[d, "N_{K/F}"]\\
				E(F)\arrow[r, "\sim"]&F^\times/q^\ZZ,
				\end{tikzcd}\]
	and so
	\[E(F)/N_{K/F}(E(K))\cong F^\times/\left(N_{K/F}(K^\times)\cdot q^\ZZ\right).\]
	Since the extension $K/F$ is unramified, local class field theory identifies the exact sequence
	
	\[\begin{tikzcd}
		0\arrow[r]&\frac{q^\ZZ}{q^\ZZ\cap N(K^\times)}\arrow[r]&\frac{F^\times}{N_{K/F}(K^\times)}\arrow[r]&\frac{F^\times}{\left(N_{K/F}(K^\times)\cdot q^\ZZ\right)}\arrow[r]&0,
	\end{tikzcd}\]
	with
	\[\begin{tikzcd}
		0\arrow[r]&\gp{v(q)}\arrow[r]&\ZZ/n\ZZ\arrow[r]&\ZZ/\gcd(v(q),n)\ZZ\arrow[r]&0
	\end{tikzcd}.\]
	The result now follows from \Cref{lem:q is delta_E}.
\end{proof}

\begin{proposition}\label{prop:nonsplit mult norm index even}
	If $E/F$ has non-split multiplicative reduction and $n\in 2\ZZ$, then 
	\[\#\left(E(F)/N_{K/F}E(K)\right)= \begin{cases}
		2&\textnormal{if }v(\Delta_E)\in 2\ZZ,\\
		1&\textnormal{else}.
	\end{cases}\]
\end{proposition}
\begin{proof}
By assumption $E$ has split multiplicative reduction over the unramified quadratic extension $L/F$, which is contained in $K$.  Write $\tau\in \gal(K/F)$ for the Frobenius element, so that $N_{K/F}=\sum_{k=0}^{n-1}\tau^k$, and $L/F$ is the fixed field of the group generated by $\tau^2$.  The Tate parametrisation of $E/K$ gives a commutative diagram
\[\begin{tikzcd}
	E(K)\arrow[d,"{N_{K/F}}"]\arrow[r, "\sim"]&K^\times/q^\ZZ\arrow[d, dashed, "\alpha"]\\
	E(F)\arrow[r, "{\sim}"]&I(F)/q^\ZZ,
\end{tikzcd}\]
where since the norm map $N_{K/F}$ factors through the field $L$ over which the Tate parametrisation is defined, the rightmost vertical map $\alpha$ is induced by the action of the element $\sum_{k=0}^{n-1}\chi_L(\tau^k)\tau^k$, where $\chi_L$ is the quadratic character cutting out the extension $L/F$.  Note that for $x\in K^\times/q^\ZZ$
\[\alpha(x)=\prod_{k=1}^{n/2}\frac{\tau^{2k}(x)}{\tau^{2k+1}(x)}=\prod_{k=1}^{n/2}\tau^{2k}\left(\frac{x}{\tau(x)}\right)=N_{K/L}\left(\frac{x}{\tau({x})}\right).\]
Thus, since by Hilbert's theorem 90 we have
\[\set{\frac{x}{\tau(x)}~:~x\in K^\times}=\ker(N_{K/F}:K^\times\to F^\times),\]
we obtain that
\begin{align*}
		E(F)/N_{K/F}(E(K))&\cong \frac{I(F)}{N_{K/L}(\ker(N_{K/F}))\cdot q^\ZZ}
		\\&\cong \frac{N_{L/F}(L^\times)\cap q^\ZZ}{q^{2\ZZ}},
\end{align*}
where since the norm map is surjective on units in unramified extensions, in particular $\ker (N_{L/F})\cap{I(F)}\subseteq N_{K/L}(\ker(N_{K/F}))$, the final isomorphism is just obtained by pushing through the map $N_{L/F}$.  It is then clear that the size of this norm index is at most $2$, and is $2$ precisely when $q$ is a norm from $L$, which by local class field theory occurs precisely when $v(q)$ is even.  The result then follows from \Cref{lem:q is delta_E}.
\end{proof}
\begin{proposition}\label{prop:nonsplit mult norm index odd}
	If $E/F$ has non-split multiplicative reduction, and $n$ is odd then
	\[\#\left(E(F)/N_{K/F}E(K)\right)=1.\]
\end{proposition}
\begin{proof}
	Let $\chi_L$ be the character associated to the unramified quadratic extension $L/F$ and write $\gal(K\cdot L/F)=\gp{\tau~:~\tau^{2n}=1}$. Letting $U$ denote units, we consider the map $f$ given by the composition
	\[\begin{tikzcd}
	U_{K\cdot L}\arrow[r, "\tilde{f}"]&I_0(K)\arrow[r, "Q"]&E_0(K),
	\end{tikzcd}\]
	where for $u\in U_{K\cdot L}$ we set $\tilde{f}(u):=\frac{u}{\tau^n(u)}$ and $Q$ is the Tate parametrisation map.  By Hilbert's theorem 90 and the fact that the extension $K\cdot L/K$ is unramified, the map $\tilde{f}$ is surjective and so since $Q$ is also surjective we must have that $f$ is a surjection.  Moreover for each $u\in U_{K\cdot L}$,
	\begin{align*}
	f(u)&=Q\left(\frac{u}{\tau^n(u)}\right)
	\\&=Q(u)-Q\left(\tau^n(u)\right)
	\\&=Q(u)-\chi_L(\tau^n)\tau^n\left(Q(u)\right)
	\\&=Q(u)+\tau^n\left(Q(u)\right)
	\\&=N_{K\cdot L/K}(Q(u)).
	\end{align*}
	Identifying $N_{K/F}=\sum_{k=0}^n\tau^{2k}$, we obtain a commutative square
	\[\begin{tikzcd}
	U_{K\cdot L}\arrow[r, "f"]\arrow[d, "N_{K/F}"]&E_0(K)\arrow[d, "N_{K/F}"]\\
				U_{L}\ar[r, "f"]&E_0(F).
				\end{tikzcd}\]
	In particular, the right hand vertical map is now a surjection since the left is by local class field theory.  This then means that $E_0(F)=N_{K/F}E_0(K)\subseteq N_{K/F}E(F)$, so in particular we have a natural surjection
	\[E(F)/E_0(F)\twoheadrightarrow E(F)/N_{K/F}E(K).\]
	Since $E$ has non-split multiplicative reduction, so has Tamagawa number $1$ or $2$, we must have that $E(F)/N_{K/F}E(K)$, which has odd order as it is a quotient of $E(F)/nE(F)$, is trivial.
\end{proof}

\noindent Our main application of the above results will be when $E/F$ has reduction type $I_1$.
\begin{lemma}\label{lem:I1reduction norm is surj}
	If $E/F$ has reduction type $I_1$, then the norm is a surjection
	\[N_{K/F}:E(K)\twoheadrightarrow E(F).\]
\end{lemma}
\begin{proof}
	By Tate's algorithm (see \Cref{lem:TatesAlg}), if $E/F$ has reduction type $I_1$ then $v(\Delta_E)=1$.  Thus the claim follows from Propositions \ref{prop:split mult norm index}, \ref{prop:nonsplit mult norm index even}, and \ref{prop:nonsplit mult norm index odd}.
\end{proof}

\section{The (Co)-Restriction Selmer Groups}\label{sec:CoresSelmer}
We will now review the properties of Selmer structures and their associated Selmer groups, before going on to extend some definitions and basic results from \cite{MorganPaterson2020}*{\S4}.  More details on Selmer structures can be found in \cites{WashingtonGaloisCohom,MazurRubinKolyvaginSystems} and the references therein. 

For the duration of this section let $F$ be a number field, $K/F$ be a finite Galois extension and $G$ be its Galois group.  Moreover, let $E/F$ be an elliptic curve and $p$ be a prime number.

\begin{definition}
	A \textit{Selmer structure} $\mathcal{L}=\{\mathcal{L}_v\}_v$ for a finite $\FF_p[G_F]$-module $M$ is a collection of subgroups
	\[\mathcal{L}_v\subseteq H^1(F_v, M),\]
	one for each $v\in \places_F$, such that $\mathcal{L}_v=H^1_{\nr}(F_v, M)$ for all but finitely many $v$.  The associated \textit{Selmer group} $\sel{\mathcal{L}}(F, M)$ is defined by the exactness of the sequence
	\[0\to \sel{\mathcal{L}}(F, M)\to H^1(F, M)\to \prod_{v\in\places_F} H^1(F_v, M)/\mathcal{L}_v.\]
	For each $v\in\places_F$ we write $\mathcal{L}_v^*$ for the orthogonal complement of $\mathcal{L}_v$ with respect to the local Tate pairing, so that $\mathcal{L}_v^*\subseteq H^1(F_v, M^*)$.  We then define the \textit{dual Selmer structure} $\mathcal{L}^*$ for $M^*$ by taking $\mathcal{L}^*=\set{\mathcal{L}_v^*}$, and refer to $\sel{\mathcal{L}^*}(F, M^*)$ as the \textit{dual Selmer group}.
\end{definition}
The following theorem describes the difference in dimension between a Selmer group and its dual.
\begin{theorem}[Greenberg--Wiles] \label{wiles greenberg formula}
Let $\mathcal{L}=\{\mathcal{L}_v\}_v$ be a Selmer structure for a finite $\FF_p[G_F]$-module $M$. Then we have
\[\dim \textup{Sel}_{\mathcal{L}}(F,M)-\dim \textup{Sel}_{\mathcal{L}^*}(F,M^*)=\dim M^{G_F}-\dim (M^*)^{G_F} +\sum_{v\in\places_F}(\dim\mathcal{L}_v- \dim M^{G_{F_v}}).\]
\end{theorem}

\begin{proof}
This follows from \cite{wiles1995modular}*{Prop 5.1(b)}.  See also \cite{WashingtonGaloisCohom}*{Theorem 2}.
\end{proof}
\begin{rem}
	Note that $E[p]$ is naturally an $\FF_p[G_F]$-module and the Weil pairing induces an $\FF_p[G_F]$-isomorphism $E[p]\cong E[p]^*$.  Making this identification, the local Tate pairing at a place $v\in\places_F$ becomes an alternating bilinear pairing on $H^1(F_v, E[p])$, and the two global terms on the right hand side of \Cref{wiles greenberg formula} cancel.
\end{rem}
We firstly define notation for the Selmer structure associated to the usual $p$-Selmer group.
\begin{definition}
	For each finite extension $L/F$ and every place $v\in \places_L$, we denote by $\SSS_v(L;E)$ the image of the coboundary map
	\[\delta_v:E(L_v)/pE(L_v)\hookrightarrow H^1(L_v, E[p]),\]
	arising from the short exact sequence of $G_{L_v}$-modules
	\[\begin{tikzcd}0\arrow[r] &E[p]\arrow[r] &E\arrow[r,"p"] &E\arrow[r] &0.\end{tikzcd}\]
	These local groups form a Selmer structure $\SSS(L)=\SSS(L;E)=\set{\SSS_v(L;E)}_v$.  Note that the associated Selmer group is $\sel{\SSS(L)}(L, E[p])=\sel{p}(E/L)$, the classical $p$-Selmer group.
\end{definition}

\begin{rem}
	The local groups $\SSS_v(L;E)$ above are in fact known to be maximal isotropic subgroups of $H^1(L_v, E[p])$ with respect to the local Tate pairing (see e.g. \cite{Poonen_2012}*{Proposition 4.10}).  In particular, $\SSS(L)$ is self-dual.
\end{rem}

\begin{definition}
	For each $v\in\places_F$ and any $w\in\places_K$ extending $v$, let 
		\[\FFF_v(K/F;E):=\res_{K_w/F_v}^{-1}(\SSS_w(K; E))\leq H^1(F_v,E[p]).\]
	Note that the definition does not depend on the choice of $w$ as our extension is Galois.  We then have a Selmer structure $\FFF(K)=\FFF(K/F;E)=\set{\FFF_v(K/F;E)}_v$ for $E[p]$ over $F$. We further define the Selmer structure $\CCC(K)=\CCC(K/F;E)$ for $E[p]$ to be the dual of $\FFF(K)$, and denote the corresponding local groups by $\CCC_v(K/F;E)$.
\end{definition}

\begin{lemma} \label{lem:inverse image of selp is selF}
We have
\[\sel{\FFF(K)}(F,E[p])=\res_{K/F}^{-1}\left(\sel{p}(E/K)\right).\]
\end{lemma}

\begin{proof}
This follows from the compatibility of local and global restriction maps. 
\end{proof}
 
\begin{lemma} \label{C is cores locally}
For every $v\in\places_F$ and every place $w\in\places_K$ extending $v$, we have 
\[\CCC_v(K/F;E)=\cores_{K_w/F_v}(\SSS_w(K;E))\leq H^1(F_v,E[p]).\]
\end{lemma}

\begin{proof}
In the case $p=2$ this is already noted by Kramer in the paragraph following Equation $10$ in \cite{kramer1981arithmetic}, and the proof in this case is explicated in \cite{MorganPaterson2020}*{proof of Lemma 4.3(i)}.  The result for general $p$ follows \textit{mutatis mutandis}.  We replicate it here for the reader's convenience.

For $v\in\places_F$ and $w\in\places_K$ extending $v$, it follows from \cite{neukirch2013class}*{I.5.4} and \cite{neukirch2013class}*{II Proposition 1.4(c) and  Theorem 5.6} that $\textup{res}_{K_w/F_v}$ and $\textup{cor}_{K_w/F_v}$ are adjoints with respect to the local Tate pairings. By \cite{Poonen_2012}*{Proposition 4.10}, $\SSS_v(F;E)$ and $\SSS_w(K;E)$ are maximal isotropic subspaces of the corresponding cohomology groups with respect to the Tate pairings.  Thus we have inclusions 
\[\cores_{K_w/F_v}(\SSS_w(K;E))\subseteq\FFF_v(K/F;E)^*\]
 and 
 \[\res_{K_w/F_v}\left(\cores_{K_w/F_v}(\SSS_w(K;E))^*\right)\subseteq \SSS_w(K;E)^*=\SSS_w(K;E).\] 
 The result then follows.
\end{proof}

We now relate the Selmer structures $\sel{\FFF(K)}(F, E[p])$ and $\sel{\CCC(K)}(F, E[p])$ to specific representation theoretic invariants of the $\FF_p [G]$-module $\sel{p}(E/F)$.

\begin{lemma}\label{lem:Relating selC with norm and selF with fixed space}
Let $N_{K/F}:=\sum_{g\in G}g\in \ZZ[G]$ be the norm element.  We have that
\begin{enumerate}[label=(\roman*)]
	\item$\dim\left(N_{K/F}\cdot \sel{p}(E/K)\right)\leq \dim\sel{\CCC(K)}(F, E[p]),\label{eq:rel selC with norm}$
	\item$\dim\left(\sel{p}(E/K)^G\right)= \dim\sel{\FFF(K)}(F, E[p])-\dim H^1(K/F, E(K)[p])+\dim(\im(\tau)),\label{eq:rel selF with fixed space}$
\end{enumerate}
where $\tau:H^1(K, E[p])\to H^2(K/F, E(K)[p])$ is the transgression map.
\end{lemma}
\begin{proof}
	\ref{eq:rel selC with norm} is given by naturality of the corestriction map, which is induced by action of $N_{K/F}$.  By \Cref{lem:inverse image of selp is selF}, the inflation-restriction sequence yields an exact sequence
	\[\begin{tikzcd}
	0\arrow[r]&H^1(K/F, E(K)[p])\arrow[r, "\infl"]&\sel{\FFF(K)}(F, E[p])\arrow[r, "\res"]&\sel{p}(E/K)^{G}\arrow[r, "\tau"]&H^2(K/F, E(K)[p]).
	\end{tikzcd}\]
	Thus \ref{eq:rel selF with fixed space} holds.
\end{proof}

We now introduce the function that will bound the failure Galois descent in our statistical results.
\begin{definition}\label{def:genus theory part}
	We define the genus theory invariant of the $p$-Selmer group of $E$ arising from the extension $K/F$ to be
	\[g_p(K/F;E):=\sum_{v\in\places_F} \dim E(F_v)/\left(N_{K_w/F_v}E(K_w)+pE(F_v)\right),\]
	where, in each summand, $w\in \Omega_K$ is any place of $K$ lying over $v$.
\end{definition}

\begin{lemma}\label{lem:selF-selC is sum of norms}We have
\[\dim \sel{\FFF(K)}(F,E[p])- \dim \sel{\CCC(K)}(F,E[p])=g_p(K/F;E), \]
and moreover,
\[0\leq\dim\sel{\FFF(K)}(F, E[p])-\dim \sel{p}(E/F)\leq g_p(K/F;E).\]
\end{lemma}
\begin{proof} 
For each $v\in\places_F$, the groups $\CCC_v=\CCC_v(K/F;E)$ and $\FFF_v=\FFF_v(K/F;E)$ are orthogonal complements under the local Tate pairing, so we have 
$\dim\FFF_v=\dim H^1(F_v,E[p])-\dim \CCC_v.$
Moreover, since $\SSS_v(F;E)$ is maximal isotropic, we have $\dim H^1(F_v,E[p])=2\dim E(F_v)/pE(F_v)$.  Combining this with \Cref{C is cores locally}, we obtain
\begin{align*}
\dim \FFF_v&= 2\dim E(F_v)/pE(F_v)-\dim\CCC_v
\\ &= 2\dim E(F_v)/pE(F_v)-\dim \left(N_{K_w/F_v}E(K_w)/\left(pE(F_v)\cap N_{K_w/F_v}E(K_w)\right)\right)
\\ &= \dim E(F_v)/pE(F_v)+\dim E(F_v)/\left(N_{K_w/F_v}E(K_w)+pE(F_v)\right).
\end{align*}
It then follows from \Cref{wiles greenberg formula} that
\begin{eqnarray*}
\dim \sel{\FFF(K)}(F,E[p])- \dim \sel{\CCC(K)}(F,E[p])&=&\sum_{v\in\places_F} \dim E(F_v)/\left(N_{K_w/F_v}E(K_w)+pE(F_v)\right) \\ &&+\sum_{v\in\places_F} \left(\dim E(F_v)/pE(F_v)- \dim E(F_v)[p]\right)
\\&=&g_p(K/F;E),
\end{eqnarray*}
where the last equality is obtained by applying \Cref{wiles greenberg formula} to the self-dual Selmer structure $\SSS(E/F)$, so the first equation in the lemma statement holds.

The second equation follows from the inclusions
\[\sel{\CCC(K)}(F, E[p])\subseteq\sel{p}(E/F)\subseteq \sel{\FFF(K)}(F, E[p]).\qedhere\]
\end{proof}

\section{Galois Descent for \texorpdfstring{$p$}{p}-Selmer Groups}\label{sec:GalDescent}
We will now use the algebraic results of \S\ref{sec:local copmputations} and \S\ref{sec:CoresSelmer} to obtain our statistical results.  This will culminate in a proof of \Cref{INTRO: gal descent failure is bdd} which tells us that, for a finite Galois extension of number fields $K/F$ and a prime number $p$, as we vary over the $E$ parametrised by $\Epsilon(X)$, the average value of
	\[\abs{\dim\sel{p}(E/K)^{\gal(K/F)}-\dim \sel{p}(E/F)},\]
which we refer to as the failure of Galois descent, is bounded as $X\to\infty$.  We use \Cref{lem:Relating selC with norm and selF with fixed space}\ref{eq:rel selF with fixed space} to relate the Selmer group $\sel{\FFF(K)}(F, E[p])$ to the Galois fixed space, which allows us to use \Cref{lem:selF-selC is sum of norms} to bound this failure of Galois descent by the genus theory invariant $g_p(K/F;E)$.  The remainder of the proof is then showing that the function $g_p(K/F;E)$ has bounded average as $E$ varies in $\Epsilon$.

\subsection{Preliminary Counting Lemmas}
We begin by recalling the description, afforded by Tate's algorithm, of the reduction type of the curves $E_{A,B}$ in terms of the pair $(A,B)\in\Epsilon$ at almost all places. 
\begin{lemma}\label{lem:TatesAlg}
	For a prime number $\ell\geq 5$ and $(A,B)\in\Epsilon$, the reduction type of $E_{A,B}/\QQ_\ell$ is
	\begin{itemize}
		\item $I_n$ for $n>0$ if and only if $v_\ell(4A^3+27B^2)=n$ and $v_\ell(AB)=0$,
		\item additive if and only if $v_\ell(\gcd(A,B))>0$,
	\end{itemize}
	where $v_\ell$ is the normalised valuation on $\QQ_\ell$.
\end{lemma}
\begin{proof}
	This is a consequence of Tate's algorithm \cite{silverman2009arithmetic}*{IV.9.4}.
\end{proof}

\begin{proposition}\label{cor:large non I1 primes}
	There exists a constant $C>0$ such that for all real numbers $X\in\mathbb{R}_{>0}$,
	\[
		\sum_{(A,B)\in \Epsilon(X)}\#\set{\ell\geq \log(X)~:~\substack{\ell\textnormal{ is prime};\\E_{A,B}/\QQ_\ell\textnormal{ has bad reduction of type}\textnormal{ different from }I_1.}}
	\leq C\left(\frac{X^{5/6}}{\log(X)}\right).\]
\end{proposition}
\begin{proof}
	We split the summand into counts of additive and multiplicative primes.
	
	By \Cref{lem:TatesAlg}, primes of additive reduction for $E_{A,B}$ divide $\gcd(A,B)$, so are bounded by the absolute values of $A$ and $B$.  Therefore, we have
	\begin{align*}
		&\sum_{(A,B)\in \Epsilon(X)}\#\set{\ell\geq \log(X)~:~\substack{\ell\textnormal{ is prime};\\E_{A,B}/\QQ_\ell\textnormal{ has additive reduction.}}}
		\\&\leq\sum_{\substack{\log(X)\leq\ell\leq X^{1/3}\\\textnormal{prime}}}\sum_{\substack{\abs{A}\leq (X/4)^{1/3}\\\ell\mid A}}\sum_{\substack{\abs{B}\leq (X/27)^{1/2}\\\ell\mid B}}1
		\\&\ll\sum_{\substack{\log(X)\leq\ell\leq X^{1/3}\\\textnormal{prime}}}\left(\frac{4X^{5/6}}{\ell^{2}}+\bigo{\frac{X^{1/2}}{\ell}}\right)
		\\&\ll \left(X^{5/6}\int_{\log(X)}^{X^{1/3}}\frac{1}{y^{2}}dy\right)+X^{1/2}\log\log(X)
		\\&\ll \frac{X^{5/6}}{\log(X)},
	\end{align*}
	where the penultimate inequality uses an integral estimate for the main term, that the sum of reciprocals of prime numbers has order $\log\log(X)$ and the prime number theorem for the error term.

	For the multiplicative primes: \Cref{lem:TatesAlg} shows that if $\ell$ is multiplicative of type different from $I_1$ for $E_{A,B}$ then $\ell^2\mid (4A^3+27B^2)$ but $\ell\nmid AB$.  Hence we have
	\begin{align*}
		&\sum_{(A,B)\in \Epsilon(X)}\#\set{\ell\geq \log(X)~:~\substack{\ell\textnormal{ is prime};\\E_{A,B}/\QQ_\ell\textnormal{ has multiplicative reduction of type}\textnormal{ different from }I_1.}}
		\\&\leq\sum_{\substack{\log(X)\leq \ell\leq \sqrt{31X}\\\textnormal{prime}}}\sum_{\substack{\abs{A}\leq (X/4)^{1/3}\\\ell\nmid A}}\sum_{\substack{\abs{B}\leq (X/27)^{1/2}\\\ell^2\mid 4A^3+27B^2}}1
		\\&\ll\sum_{\substack{\log(X)\leq \ell\leq \sqrt{31X}\\\textnormal{prime}}}\left(\frac{X^{5/6}}{\ell^2}+\bigo{X^{1/3}}\right)
		\\&\ll \frac{X^{5/6}}{\log(X)}.
	\end{align*}
	The result follows.
\end{proof}
\subsection{Bounding the Genus Theory Invariant}
We begin by noting some elementary bounds on the norm indices which occur as summands in the genus theory invariant (as in \Cref{def:genus theory part}).

\begin{lemma}\label{lem:norm indice bounds}
	Let $F$ be a number field, $K/F$ be a finite extension, $p$ be a prime number and $E/F$ be an elliptic curve.  
		\begin{equation}
		\dim E(F_v)/pE(F_v)
		\leq\begin{cases}
		2+[F_v:\QQ_p]&\textnormal{if }v\mid p,\\
		2&\textnormal{if }v\textnormal{ is a finite place and }v\nmid p,\\
		1&\textnormal{if }v\textnormal{ is a real place and }p=2,\\
		0&\textnormal{otherwise.}
		\end{cases}\end{equation}
		In particular, for every $v\in\places_F$ and each $w\in\places_K$ extending $v$, the same bound holds for $\dim E(F_v)/\left(N_{K_w/F_v}E(K_w)+pE(F_v)\right)$.
\end{lemma}
\begin{proof}
For each finite place $\pp\in\places_F$ and each $E/\QQ$, there is a finite index subgroup, arising from the filtration by formal groups, of $E(F_\pp)$ which is isomorphic to the additive group of integers $\OO_\pp$ of $F_\pp$ (see e.g. \cite{silverman2009arithmetic}*{VII Prop. 6.3}).  Thus these norm indices are bounded by
	\begin{equation}
	\#E(F_\pp)/pE(F_\pp)=(\#E(F_\pp)[p])(\# \OO_{\pp}/p\OO_{\pp})\leq\begin{cases}
	p^{2+[F_\pp:\QQ_p]}&\pp\mid p,\\
	p^2&\textnormal{else}.
	\end{cases}\end{equation}
	Moreover, for archimedean places $v\in\places_F$, if $p$ is odd or $v$ is complex then we have $\dim E(F_v)/pE(F_v)\leq \dim H^1(F_v, E[p])=0$.  If, on the other hand, $p=2$ and $v$ is real then elementary computations show that the dimension of the quotient at $v$ is at most $1$.
\end{proof}
We are now mathematically ready to bound the average of the genus theory invariant, but first we require a small amount of notation.
\begin{notation}
	For a number field $F$, we define the function $\omega_F$ on the set of ideals of the integers of $F$ to send the ideal $I$ to
	\[\omega_F(I):=\#\set{\pp\in\places_F~:~\pp\mid I}.\]
	We also define $r_1(F)$ to be the number of real embeddings of $F$.  Moreover, $\delta_2$ is the function which takes each prime number $p$ to $1$ if $p=2$ and $0$ otherwise.
\end{notation}
We now bound the average of the genus theory invariant.
\begin{proposition}\label{prop:genus theory is bounded}
	For every number field $F$, finite Galois extension $K/F$, prime number $p$ and real number $X\in\RR_{>0}$ we have
	\[\frac{\sum\limits_{(A,B)\in\Epsilon(X)}g_p(K/F;E_{A,B})}{\#\Epsilon(X)}\leq C_p(K/F)+\bigo{\frac{[F:\QQ]}{\log(X)}},\]
	where
	\[C_p(K/F)=2 \omega_F(6p\Delta_{K})+[F:\QQ]+\delta_2(p)r_1(F) +
		2\sum\limits_{\substack{\ell \textnormal{ prime}\\\ell\nmid 6p\Delta_{K}}}\omega_F(\ell)\frac{2\ell^8-\ell^7-1}{\ell^{10}-1}.\]
\end{proposition}
\begin{proof}
	For each elliptic curve $E/\QQ$, number field $F$, and finite Galois extension $K/F$, define
	\begin{align*}
		g^{(0)}_p(K/F;E)&=\sum_{\substack{v\in\places_F\\v\mid 6p\infty\Delta_{K}}}\dim E(F_v)/\left(N_{K_w/F_v}E(K_w)+pE(F_v)\right),
		\\g^{(1)}_p(K/F;E)&=\sum_{\substack{\pp\in\places_F\\\pp\nmid 6p\infty\Delta_{K}\\\pp\mid N(E/F)}}\dim E(F_\pp)/\left(N_{K_\PP/F_\pp}E(K_\PP)+pE(F_\pp)\right),
	\end{align*}
	where in each summand, $w$ (resp. $\PP$) is a place of $K$ above $v$ (resp. $\pp$), and $N(E/F)$ is the conductor of $E/F$.  By \cite{mazur1972rational}*{Corollary 4.4}, the norm map is surjective at primes of good reduction which are unramified in $K/F$, so the norm indices at such primes are trivial.  Thus
	\[g_p(K/F;E)=g_p^{(0)}(K/F;E)+g_p^{(1)}(K/F;E),\]
	so we bound the average of $g_p^{(i)}(K/F;E)$ for $i\in\set{0,1}$.

	If $i=0$ then by \Cref{lem:norm indice bounds},
	\begin{align*}
		\sum\limits_{(A,B)\in\Epsilon(X)}g_p^{(0)}(K/F;E_{A,B})&\leq \left(2\#\set{\pp\in\places_F~:~\pp\mid 6p\Delta_{K}}+[F:\QQ]+\delta_2(p)r_1(F)\right)\#\Epsilon(X).
	\end{align*}

	We now deal with the case that $i=1$.  By \Cref{lem:I1reduction norm is surj}, the norm index at primes of reduction type $I_1$ is trivial.  Thus, for each elliptic curve $E/\QQ$, the sum $g_p^{(1)}(K/F;E)$ is the sum of norm indices at unramified primes of bad reduction of type different from $I_1$ over $F$. By the methods of \cite{cho2020distribution}*{Theorem 1.4}, which work identically for our height as for theirs, for each prime number $\ell\in [5, X^{1/6}]$ one has
	\begin{equation}\label{eq:curves with bad redux not I1}
	\#\set{(A,B)\in\Epsilon(X)~:~\substack{E_{A,B}/\QQ_\ell\textnormal{ has bad reduction}\\\textnormal{ of type different from }I_1}}=\#\Epsilon(X)\frac{2\ell^8-\ell^7-1}{\ell^{10}-1}+\bigo{\ell X^{1/2}}.
	\end{equation}

	Since we are looking at unramified local extensions $F_\pp/\QQ_\ell$, curves with bad reduction of type different from $I_1$ over $F_\pp$ must satisfy the same condition over $\QQ_\ell$. We then have 
	\begin{align*}
		&\sum_{(A,B)\in\Epsilon(X)}g_p^{(1)}(K/F;E_{A,B})
		\\&\leq 2\sum_{\substack{5\leq \ell\leq 31X\\\textnormal{prime}\\\ell\nmid p\Delta_{K}}}\sum_{\substack{\pp\in\places_F\\\pp\mid \ell}}\#\set{(A,B)\in\Epsilon(X)~:~\substack{E_{A,B}\textnormal{ has bad reduction}\\\textnormal{of type different from }I_1\textnormal{ at }\ell}}
		\\&\leq 2\sum_{\substack{5\leq \ell\leq \log(X)\\\textnormal{prime}\\\ell\nmid p\Delta_{K}}}\sum_{\substack{\pp\in\places_F\\\pp\mid \ell}}\left(\#\Epsilon(X)\frac{2\ell^8-\ell^7-1}{\ell^{10}-1}+\bigo{\ell X^{1/2}}\right)+\bigo{\frac{X^{5/6}[F:\QQ]}{\log(X)}}
		\\&\leq 2\#\Epsilon(X)\sum\limits_{\substack{\ell \textnormal{ prime}\\\ell\nmid 6p\Delta_{K}}}\#\set{\pp\in\places_F~:~\pp\mid \ell}\frac{2\ell^8-\ell^7-1}{\ell^{10}-1}+\bigo{\frac{X^{5/6}[F:\QQ]}{\log(X)}},
	\end{align*}
	where in the first inequality we bound the norm index by \Cref{lem:norm indice bounds}, and in the second we discount large primes using \Cref{cor:large non I1 primes} and then apply \eqref{eq:curves with bad redux not I1}.  The bound then follows from the well known fact that $\#\Epsilon(X)\sim cX^{5/6}$ for some $c>0$.
\end{proof}

\subsection{Proof of \texorpdfstring{\Cref{INTRO: gal descent failure is bdd}}{Average Failure of Galois Descent Being Bounded}}
We first use the Selmer structures of \S\ref{sec:CoresSelmer} to approximate the dimension of the corresponding fixed space.  To begin, almost no elliptic curves defined over $\QQ$ have nontrivial $n$-torsion over a fixed number field $K$.  The proof of this is obtained verbatim from the argument of Duke \cite{DUKE1997813}*{Lemma 5} in the case $K=\QQ$, applying the relevant sieve conditions only at totally split primes as performed by Zywina \cite{Zywina_2010}*{Proposition 5.7}.
\begin{lemma}\label{prop:torsion bound}
	Let $n$ be a positive integer and let $K/\QQ$ be a finite extension.  Then
	\[\frac{\#\set{(A,B)\in\Epsilon(X)~:~E_{A,B}(K)[n]\textnormal{ is nontrivial}}}{\#\Epsilon(X)}\ll_{n,K} \frac{\log(X)}{X^{1/6}}.\]
\end{lemma}

Using this result, we can prove the following.

\begin{lemma}\label{lem:fixed space approximates sel F statistically}
	Let $p$ be a prime number, $F$ be a number field and $K/F$ be a finite Galois extension.  We have that
	\[\frac{\sum\limits_{(A,B)\in\Epsilon(X)}\abs{\dim\sel{p}(E_{A,B}/K)^G - \dim\sel{\FFF(K)}(F, E[p])}}{\#\Epsilon(X)}\ll_{K,p}\frac{\log(X)}{X^{1/6}},\]
	where $G=\gal(K/F)$ is the Galois group.
\end{lemma}
\begin{proof}
	Let $D_p(G)$ be a positive integer such that, for every $\FF_p[G]$-module $M$ of dimension at most $2$ and every $i\in\set{1,2}$, we have
	\[\dim H^i(G, M)\leq D_p(G).\]
	Since there are only finitely many such $M$, $D_p(G)$ certainly exists.  By \Cref{lem:Relating selC with norm and selF with fixed space}, for every elliptic curve $E/\QQ$ we have
	\[\abs{\dim\sel{p}(E_{A,B}/K)^G - \dim\sel{\FFF(K)}(F, E[p])}\leq \begin{cases}
		0 & \textnormal{if }E(K)[p]\textnormal{ is trivial,}\\
		D_p(G) &\textnormal{ else.}
	\end{cases}\]
	The result then follows from \Cref{prop:torsion bound}.
\end{proof}

We now combine this with \Cref{prop:genus theory is bounded} to prove \Cref{INTRO: gal descent failure is bdd}, namely that the average failure of Galois descent is bounded.

\begin{theorem}\label{thm:avg dim of fixed space is bdd}
	Let $p$ be a prime number, $F$ be a number field and $K/F$ be a finite Galois extension.  Writing $G=\gal(K/F)$, we have that
	\[\limsup_{X\to\infty}\frac{\sum\limits_{(A,B)\in\Epsilon(X)}\abs{\dim\sel{p}(E_{A,B}/K)^G - \dim\sel{p}(E_{A,B}/F)}}{\#\Epsilon(X)}\leq C_p(K/F),\]
	where $C_{p}(K/F)$ is as in \S\ref{subsec:bound shape}.
\end{theorem}
\begin{proof}
	By \Cref{lem:fixed space approximates sel F statistically}, we immediately have 
	\begin{align*}
		&\limsup_{X\to\infty}\frac{\sum\limits_{(A,B)\in\Epsilon(X)}\abs{\dim\sel{p}(E_{A,B}/K)^G - \dim\sel{p}(E_{A,B}/F)}}{\#\Epsilon(X)}
		\\&\leq \limsup_{X\to\infty}\frac{\sum\limits_{(A,B)\in\Epsilon(X)}\abs{\dim\sel{\FFF(K)}(F, E_{A,B}[p]) - \dim\sel{p}(E_{A,B}/F)}}{\#\Epsilon(X)}.
	\end{align*}
	Since by \Cref{lem:selF-selC is sum of norms} this average is bounded by that of the genus theory invariant, the result follows from \Cref{prop:genus theory is bounded}.
\end{proof}
From this we derive an immediate consequence.

\begin{cor}\label{cor:FixedSpaceBoundedDim}
	Let $p\in\set{2,3,5}$ and let $K/\QQ$ be a finite Galois extension.  Then, writing $G=\gal(K/\QQ)$, we have
	\[\limsup_{X\to\infty}\frac{\sum\limits_{(A,B)\in\Epsilon(X)}\dim\sel{p}(E_{A,B}/K)^G}{\#\Epsilon(X)}\leq  C_p(K/\QQ)+\frac{p+1}{p},\]
	where $C_{p}(K/\QQ)$ is as in \S\ref{subsec:bound shape}.  Assuming \Cref{hyp:PR for cong conds} the same is true if $p$ is any prime number.
\end{cor}
\begin{proof}
	This follows from \Cref{thm:avg dim of fixed space is bdd} and \Cref{prop:Bharghava Shankar}.
\end{proof}

\begin{example}\label{ex:S3 extension fixed space}
	Consider the splitting field $K/\QQ$ of $x^3-2$, which is a degree $6$ extension with Galois group $G\cong S_3$.

	If $p=2$, it follows from \Cref{cor:FixedSpaceBoundedDim} that the average dimension of $\sel{2}(E/K)^G$ is at most $C_2(K/\QQ)+\frac{3}{2}$.  The primes dividing $6p\Delta_{K}$ are $2$ and $3$, so that 
	\[C_2(K/\QQ)=6+2\sum_{\substack{\ell\neq 2,3\\\textnormal{prime}}}\frac{2\ell^8-\ell^7-1}{\ell^{10}-1}\approx 6.339.\]
	Thus, the average of $\dim\sel{2}(E/K)^G$ is less than $7.839$.

	Similarly, if $p=3$, the average of $\dim\sel{3}(E/K)^G$ is less than $6.672$.

	For every prime number $p$ different from $2$ and $3$, and every elliptic curve $E/\QQ$, we have that $\sel{p}(E/K)^G\cong\sel{p}(E/\QQ)$ by \Cref{prop:Selmer Descent splits in good char} (one can also note this by the vanishing of the finite group cohomology in the inflation restriction sequence).  In particular, for $p=5$ the average of the dimension of this fixed space is at most $6/5$ by \cite{bhargava5Sel}, and for the remaining $p$ is predicted by the Poonen--Rains heuristics.
\end{example}

\section{Boundedness of Selmer Ranks}\label{sec:Selmer Ranks}
In this section we use the modular representation theory of $p$--groups to leverage the result of \Cref{thm:avg dim of fixed space is bdd} to obtain a bound for the average dimension of the entire $p$--Selmer group, not just that of the fixed space.  Combining this with estimates for $p$--Selmer groups over multiquadratic extensions from \Cref{prop:multiquadratic good char average} we then prove explicit upper bounds for average $p$--Selmer ranks over Galois $p$--extensions of $\QQ$ and of multiquadratic number fields.

\subsection{\texorpdfstring{$p$}{p}-Selmer Ranks for \texorpdfstring{$p$}{p}-Extensions}
The modular representation theory of groups of prime order is well known, we recall it below.
\begin{lemma}\label{lem:modular representation theory of cyclic p-gps}
	Let $p$ be a prime number, and $G$ be a cyclic group of order $p$.  The isomorphism classes of finitely generated indecomposable $\FF_p[G]$-modules are represented precisely by $\set{M_k}_{k=1}^p$, where $M_1$ is the 1-dimensional vector space $\FF_p$ with trivial $G$-action and $M_k$ is a non-split extension of $M_{k-1}$ by $M_1$.  Moreover, every $\FF_p[G]$-module is isomorphic to a unique direct sum of these indecomposable modules.
\end{lemma}
\begin{proof}
	By the orbit-stabiliser theorem we have that there is precisely one simple $\FF_p[G]$-module, the trivial module $M_1$.  The result then follows from the Krull-Schmidt theorem and the existence of Jordan normal form (see, for example, \cite{alperin_1986}*{page 24}).
\end{proof}

This will be sufficient to extend the boundedness result to the full $p$-Selmer group.

\begin{theorem}\label{thm:pSel from p ext is bdd}
	Let $p$ be a prime number, $F$ be a number field and $K/F$ be a Galois $p$--extension.  Then
	\begin{eqnarray*}
	\lefteqn{\limsup_{X\to\infty}\frac{\sum\limits_{(A,B)\in\Epsilon(X)}\dim\sel{p}(E_{A,B}/K)}{\#\Epsilon(X)}}
	\\&&\leq [K:F]\left(C_p(K/F)+\limsup_{X\to\infty}\frac{\sum\limits_{(A,B)\in\Epsilon(X)}\dim\sel{p}(E_{A,B}/F)}{\#\Epsilon(X)}\right),
	\end{eqnarray*}
	where $C_{p}(K/F)$ is as in \S\ref{subsec:bound shape}.
\end{theorem}
\begin{proof}
	Write $[K:F]=p^k$ for some integer $k>0$.  As $G$ is soluble, we let $F=L_0\subseteq L_1\subseteq \dots\subseteq L_k=K$ be intermediate subfields such that for each $i\in\set{1,\dots,k}$ we have $\gal(L_{i}/L_{i-1})\cong \ZZ/p\ZZ$.  We have for each such $i$ that by \Cref{lem:modular representation theory of cyclic p-gps} there are precisely $p$ indecomposable $\FF_p[\gal(L_{i}/L_{i-1})]$-modules, each of which is given by mapping a generator to a Jordan block of length between $1$ and $p$.  Hence, for every elliptic curve $E/\QQ$, we have an inequality
	\[\dim\sel{p}(E/K)\leq p\dim\sel{p}(E/K)^{\gal(K/L_{k-1})}.\]
	Moreover, since $\sel{p}(E/K)^{\gal(K/L_{k-1})}$ is an $\FF_p[\gal(L_{k-1}/L_{k-2})]$-module we again obtain
	\begin{align*}
	\dim\sel{p}(E/K)^{\gal(K/L_{k-1})}
	&\leq p\dim\left(\sel{p}(E/K)^{\gal(K/L_{k-1})}\right)^{\gal(L_{k-1}/L_{k-2})}
	\\&=p\dim\sel{p}(E/K)^{\gal(K/L_{k-2})}.
	\end{align*}
	Continuing, we obtain
	\[\dim\sel{p}(E/K)\leq p^k\dim\sel{p}(E/K)^{\gal(K/F)},\]
	so the result follows from \Cref{thm:avg dim of fixed space is bdd}.
\end{proof}

We can then combine the bound in \Cref{thm:pSel from p ext is bdd} with the bound already established in \Cref{prop:multiquadratic good char average} to obtain the full statement of \Cref{INTRO: full p selmer thm} and so \Cref{INTRO RANK THM} via the inclusion $E(K)/pE(K)\subseteq \sel{p}(E/K)$.
\begin{cor}\label{cor:p extensions of multiquadratics}
	Let $p\in\set{2,3,5}$, $F$ be either $\QQ$ or a multiquadratic number field, and $K/F$ be a Galois $p$--extension.  Then
	\begin{eqnarray*}
	\lefteqn{\limsup_{X\to\infty}\frac{\sum\limits_{(A,B)\in\Epsilon(X)}\dim\sel{p}(E_{A,B}/K)}{\#\Epsilon(X)}}
	\\&&\leq
	\begin{cases}
	[K:F]C_2(K/F)+[K:\QQ]\left(C_2(F/\QQ)+\frac{3}{2}\right)&\textnormal{if }p=2\textnormal{ and }F\neq\QQ,\\
	[K:F]\left(C_p(K/F)+\frac{p+1}{p}[F:\QQ]\right)&\textnormal{else,}
	\end{cases}
	\end{eqnarray*}
	where $C_{p}(K/F)$ is the explicit constant in \S\ref{subsec:bound shape}.  Moreover, assuming \Cref{hyp:PR for cong conds} the same is true if $p$ is any prime number.
\end{cor}
\begin{proof}
	If $p$ is odd, then this is immediate from \Cref{thm:pSel from p ext is bdd} and \Cref{prop:multiquadratic good char average}.  If both $p=2$ and $F=\QQ$ then it is immediate from \Cref{thm:pSel from p ext is bdd} and \Cref{prop:Bharghava Shankar}.  If $p=2$ and $F$ is a multiquadratic extension, then we apply \Cref{thm:pSel from p ext is bdd} twice: first to the extension $K/F$, then to $F/\QQ$, since both are Galois $2$--extensions.  The result in this case then follows from \Cref{prop:Bharghava Shankar}.
\end{proof}

\section{Mordell--Weil Lattices over Galois Extensions}
\label{sec:MW app}
\subsection{Mordell--Weil Lattices}
Our main object of study here will be the Mordell--Weil lattice, which is the ``free part'' of the Mordell-Weil group.
\begin{definition}
For a number field $K$ and an elliptic curve $E/K$, the Mordell--Weil lattice is the quotient
	\[\Lambda(E/K):=E(K)/E(K)_{\tors}.\]
\end{definition}
When $K/F$ is a Galois extension of number fields and $E$ is defined over $F$, this is evidently a finitely generated $\ZZ[\gal(K/F)]$-module which is free as a $\ZZ$-module.  We refer to such modules as $\ZZ[\gal(K/F)]$-lattices.  We begin by giving a precise notion of ``multiplicity'' for indecomposable lattices in Mordell--Weil lattices.
\begin{definition}\label{def:multiplicity}
	Let $p$ be a prime number, $K/F$ be a finite Galois extension of number fields and $E/F$ be an elliptic curve. For each finitely generated $\ZZ[\gal(K/F)]$
	-lattice $\Lambda$, define the multiplicity of $\Lambda$ in $E(K)$ to be
	\[e_{\Lambda}(K/F;E):=\max \set{e\in\ZZ_{\geq 0}~:~\Lambda^{\oplus e}\textnormal{ is a direct summand of }\Lambda(E/K)\textnormal{ as }\ZZ[\gal(K/F)]\textnormal{-lattices}}.\]
\end{definition}
\begin{example}
	Let $K/\QQ$ be the splitting field of the polynomial $x^3-3x-1$.  Note that $K/\QQ$ is Galois and has degree $3$, and write $G=\gal(K/\QQ)$.  There are two irreducible $\QQ[G]$-modules: the line $\QQ$, with trivial $G$-action, and the third cyclotomic field $\QQ(\zeta_3)$, where a generator of $G$ acts by multiplication by $\zeta_3$.  Moreover, Maschke's theorem tells us that finite dimensional $\QQ[G]$-modules are semisimple, so are isomorphic to direct sums of these irreducible modules.

	Let $E/\QQ$ be the elliptic curve described by the Weierstrass equation
	\[E:y^2 + xy = x^3 - x^2 - 42x - 19.\]
	The computer algebra program \texttt{MAGMA} \cite{MR1484478} can compute that $E(K)$ is torsion--free of rank $2$ and $E(\QQ)$ is trivial.   Since there are no points fixed by the Galois action, $e_{\ZZ}=0$ where $\ZZ$ is the set of integers acted on trivially by $G$.  Moreover, $E(K)\otimes\QQ\cong \QQ(\zeta_3)$, so the Mordell--Weil group is isomorphic to a $\ZZ[\zeta_3]$-stable lattice inside of $\QQ(\zeta_3)$.
	Such lattices are precisely the fractional ideals, and since scaling such a lattice gives an isomorphic module and the class group of $\QQ(\zeta_3)$ is trivial, $\Lambda(E/K)=E(K)$ is isomorphic to $\ZZ[\zeta_3]$ as $\ZZ[G]$-lattices.  In particular, $e_{\ZZ[\zeta_3]}(E/K)=1$.
\end{example}

We shall give upper bounds for the averages of some of these exponents by considering the lattice modulo $p$, and then estimating the various exponents in terms of the fixed space in the $p$-Selmer group.
\begin{lemma}\label{prop:eLambda precursor}
	Let $p$ be a prime number, $K/F$ be a finite Galois extension of number fields, and $E/F$ be an elliptic curve.  Writing $G=\gal(K/F)$, we have that
	\[\dim(\Lambda(E/K)/p\Lambda(E/K))^G\leq \dim\sel{p}(E/K)^G - \dim E(F)[p]+\dim H^1(G, E(K)[p]).\]
\end{lemma}
\begin{proof}
	Since $E(K)[p^\infty]/pE(K)[p^\infty]\cong E(K)[p]$ as $\ZZ[G]$--modules, there is a short exact sequence of $\FF_p[G]$-modules
	\[\begin{tikzcd}
		0\arrow[r]&E(K)[p]\arrow[r]&E(K)/pE(K)\arrow[r]&\Lambda(E/K)/p\Lambda(E/K)\arrow[r]&0,
	\end{tikzcd}\]
	so that, taking cohomology over $G$, we obtain
	\[\dim (\Lambda(E/K)/p\Lambda(E/K))^G\leq \dim (E(K)/pE(K))^G-\dim E(F)[p]+\dim H^1(G, E(K)[p]). \]
	Moreover, the short exact sequence induced by multiplication by $p$ gives an inclusion of $\FF_p[G]$-modules
	\[\delta:E(K)/pE(K)\hookrightarrow\sel{p}(E/K),\]
	completing the result.
\end{proof}
\begin{proposition}\label{prop:eLambda in terms of selmer fixed space}
	Let $p$ be a prime number, $K/F$ be a finite Galois extension of number fields, and $E/F$ be an elliptic curve.  Writing $G=\gal(K/F)$, then for every $\ZZ[G]$-lattice $\Lambda$ such that $\dim(\Lambda/p\Lambda)^G\geq 1$, we have that
	\[e_\Lambda(K/F;E)\leq \frac{1}{\dim (\Lambda/p\Lambda)^G}\left(\dim\sel{p}(E/K)^G - \dim E(F)[p]+\dim H^1(G, E(K)[p])\right).\]
\end{proposition}
\begin{proof}
	If $\Lambda^{\oplus e}$ is a direct summand of $\Lambda(E/K)$, then
	\[\left((\Lambda/p\Lambda)^G\right)^{\oplus e}\subseteq \left(\Lambda(E/K)/p\Lambda(E/K)\right)^G,\]
	so that, since $\dim\Lambda/p\Lambda^G\geq 1$, we have
	\begin{equation}\label{eq:elambda bounded by lattice data}
	e_\Lambda(K/F;E)\leq \frac{\dim(\Lambda(E/K)/p\Lambda(E/K))^G}{\dim (\Lambda/p\Lambda)^G}.
	\end{equation}
	Thus the result follows from \Cref{prop:eLambda precursor}.
\end{proof}
\subsection{Average Multiplicities}
We now use \Cref{thm:avg dim of fixed space is bdd} to obtain the average multiplicity of certain lattices in Mordell--Weil lattices of elliptic curves.

\begin{theorem}\label{thm:general p MW result}
	Let $K/F$ be a finite Galois extension of number fields, write $G=\gal(K/F)$ and let $p$ be a prime number.  For every $\ZZ [G]$-lattice $\Lambda$ such that $\dim(\Lambda/p\Lambda)^G\geq 1$,
	\begin{eqnarray*}
	\lefteqn{\limsup_{X\to\infty}\frac{\sum\limits_{(A,B)\in\Epsilon(X)}e_\Lambda(K/F;E_{A,B})}{\#\Epsilon(X)}}
	\\&&\leq \frac{1}{\dim (\Lambda/p\Lambda)^G} \left(C_p(K/F)+\limsup_{X\to\infty}\frac{\sum\limits_{(A,B)\in\Epsilon(X)}\dim\sel{p}(E_{A,B}/F)}{\#\Epsilon(X)}\right),
	\end{eqnarray*}
	where $C_{p}(K/F)$ is as in \S\ref{subsec:bound shape}.
\end{theorem}
\begin{proof}
	Let $D_p(G)$ be an integer such that for every elliptic curve $E/\QQ$ we have that 
	\[\dim H^1(G, E(K)[p])-\dim E(F)[p]\leq D_p(G).\]
	Note that this exists, since there are only finitely many $\FF_p[G]$-modules of dimension at most $2$.  Now, by \Cref{prop:torsion bound}
	\[\frac{\sum_{(A,B)\in\Epsilon(X)}\left(\dim H^1(G, E(K)[p])-\dim E(F)[p]\right)}{\#\Epsilon(X)}\ll_{K,p}D_p(G)\frac{\log(X)}{X^{1/6}},\]
	and the result follows from \Cref{prop:eLambda in terms of selmer fixed space} and \Cref{thm:avg dim of fixed space is bdd}.
\end{proof}

\begin{rem}\label{rem:how to check that lambda/p has fixed pts}
	The requirement that $(\Lambda/p\Lambda)^G$ is non-trivial for some prime number $p$ is rather easy to check.  If $\Lambda^G\neq 0$ then already this is non-trivial for every prime number, and if $\Lambda^G=0$ then via the short exact sequence induced by multiplication by $p$, $(\Lambda/p\Lambda)^G$ is isomorphic to the $p$-torsion of the finite cohomology group $H^1(G, \Lambda)$.  Computing this cohomology group in any given instance is a purely mechanical task.
\end{rem}

We then immediately obtain \Cref{INTRO:general p MWlattice result}.
\begin{cor}\label{cor:p=235 general MW result}
	Let $p\in\set{2,3,5}$, $F$ be either $\QQ$ or a multiquadratic number field, and $K/F$ be a finite Galois extension.  Write $G=\gal(K/F)$, then for every $\ZZ [G]$-lattice $\Lambda$ such that $\dim(\Lambda/p\Lambda)^G\geq 1$,
	\begin{eqnarray*}
	\lefteqn{\limsup_{X\to\infty}\frac{\sum\limits_{(A,B)\in\Epsilon(X)}e_\Lambda(K/F;E_{A,B})}{\#\Epsilon(X)}}
	\\&&\leq\frac{1}{\dim(\Lambda/p\Lambda)^G}\cdot
	\begin{cases}
	C_2(K/F)+[F:\QQ]\left(C_2(F/\QQ)+\frac{3}{2}\right)&\textnormal{if }p=2\textnormal{ and }F\neq\QQ,\\
	C_p(K/F)+\frac{p+1}{p}[F:\QQ]&\textnormal{else,}
	\end{cases}
	\end{eqnarray*}
	where $C_{p}(K/F)$ is the explicit constant in \S\ref{subsec:bound shape}.  Moreover, under \Cref{hyp:PR for cong conds} the same is true if $p$ is any prime number.
\end{cor}

\begin{proof}
	Applying \Cref{thm:general p MW result} and \Cref{prop:torsion bound}, it is sufficient to replace the numerator in the left hand side with $\dim\sel{p}(E_{A,B}/F)$ and bound the average appropriately in each case.

	If $p\in\set{3,5}$, then this follows from \Cref{prop:multiquadratic good char average}; if $p=2$ and $F=\QQ$ then it follows from \Cref{prop:Bharghava Shankar}; and finally, if $p=2$ and $F$ is a multiquadratic number field then it follows from \Cref{cor:p extensions of multiquadratics}.
\end{proof}
\subsection{An Example: Semidirect Products}
We conclude by providing a family of examples of lattices which satisfy the hypotheses of \Cref{thm:general p MW result} and generalise \Cref{ex:MWLatticeExample} from the introduction.  Let $K/\QQ$ be a finite Galois extension such that $G=\gal(K/\QQ)$ is an inner semidirect product $N\rtimes H$.  Consider the augmentation ideal $\Lambda\subseteq \ZZ[N]$, which is defined by the short exact sequence of $\ZZ[N]$--modules:
\begin{equation}\label{eq:ExampleOfSemiDirectLatticesSESatN}
\begin{tikzcd}
	0\arrow[r]&\Lambda\arrow[r]&\ZZ[N]\arrow[r,"\epsilon"]& \ZZ\arrow[r]&0,
\end{tikzcd}
\end{equation}
where the augmentation map $\epsilon$ is given explicitly by $\sum_{n\in N}a_n\cdot n\mapsto \sum_{n\in N}a_n$.  

Identifying each $n\in N$ with the coset $nH\in G/H$ provides an isomorphism of $\ZZ[N]$--modules $\ZZ[N]\cong \ZZ[G/H]$.  This identification allows us to induce a $G$--action on $\Lambda\subseteq \ZZ[G/H]$, and to upgrade \eqref{eq:ExampleOfSemiDirectLatticesSESatN} to a short exact sequence of $\ZZ[G]$--modules.  Taking cohomology over $N$ we obtain an exact sequence of $\ZZ[G/N]$--modules
\begin{equation}\label{eq:ExampleOfSemiDirectLatticesSESForNCohom}
\begin{tikzcd}
	0\arrow[r]&\Lambda^N\arrow[r]&\ZZ[G/H]^N\arrow[r,"\epsilon"]& \ZZ\arrow[r]&H^1(N, \Lambda)\arrow[r]&0.
\end{tikzcd}
\end{equation}
In particular, as $\ZZ[G/H]^N=\ZZ\cdot (\sum_{n\in N}nH)$ so that $\epsilon$ is injective on the fixed points, we have that $\Lambda^N=0$.  By \Cref{rem:how to check that lambda/p has fixed pts}, since $\Lambda^G\subseteq \Lambda^N=0$, we have that for every prime number $p$
\[(\Lambda/p\Lambda)^G\cong H^1(G, \Lambda)[p].\]

It follows from the inflation restriction short exact sequence that $H^1(G, \Lambda)\cong H^1(N, \Lambda)^{G/N}$.  Again considering \eqref{eq:ExampleOfSemiDirectLatticesSESForNCohom}, we have that $H^1(N,\Lambda)\cong \ZZ/\#N\ZZ$ with trivial $G/N$--action.  In particular, for all primes $p\mid \#N$ we have that
\[(\Lambda/p\Lambda)^G\cong \ZZ/p\ZZ.\]

Thus, if $\#N$ is divisible by $2,3$ or $5$ then by \Cref{cor:p=235 general MW result} we have that the average of $e_\Lambda(K/\QQ;E)$ is bounded as $E/\QQ$ runs through elliptic curves ordered by height.  Moreover, assuming \Cref{hyp:PR for cong conds} the same is true for any nontrivial $N$.
\bibliographystyle{plain}
\bibliography{refs}
\end{document}